\documentclass{amsart}
\usepackage{graphics}
\usepackage{amssymb}
\usepackage{amscd}
\usepackage{amsmath}
\usepackage{enumerate}
\usepackage[dvips]{graphicx}
\usepackage[all]{xy}
\usepackage{mathrsfs}
\newtheorem{theo}{Theorem}
\newtheorem{lema}[theo]{Lemma}

\newtheorem{prop}[theo]{Proposition}

\newtheorem{definition}[theo]{Definition}

\newtheorem{remark}[theo]{Remark}
\newtheorem{notation}[theo]{Notation}
\newtheorem{example}[theo]{Example}

\newcommand{\CC}{{\mathbb{C}}}
\newcommand{\NN}{{\mathbb{N}}}
\newcommand{\PP}{{\mathbb{P}}}

\newcommand{\RR}{{\mathbb{R}}}

\newcommand{\SSS}{{\mathbb{S}}}
\newcommand{\ZZ}{{\mathbb{Z}}}

\newcommand{\calA}{{\mathcal{A}}}
\newcommand{\calB}{{\mathcal{B}}}
\newcommand{\calC}{{\mathcal{C}}}
\newcommand{\calD}{{\mathcal{D}}}

\newcommand{\calF}{{\mathcal{F}}}

\newcommand{\calM}{{\mathcal{M}}}

\newcommand{\calO}{{\mathcal{O}}}

\newcommand{\calS}{{\mathcal{S}}}

\newcommand{\calU}{{\mathcal{U}}}
\newcommand{\calV}{{\mathcal{V}}}

\newcommand{\calX}{{\mathcal{X}}}
\newcommand{\calY}{{\mathcal{Y}}}
\newcommand{\calZ}{{\mathcal{Z}}}
\newcommand{\comp}{{\circ}}

\newcommand{\fraX}{{\mathcal{X}}}
\begin{document}
\title[Bouquet theorem]{Topology of hypersurface singularities with 3-dimensional critical set}
\author{Javier Fern\'andez de Bobadilla}
\author{Miguel Marco-Buzun\'ariz}
\address{ICMAT. CSIC-Complutense-Aut\'onoma-Carlos III}
\email{javier@mat.csic.es}
\email{mmarco@unizar.es}
\thanks{Research partially supported by the ERC Starting Grant project TGASS and by Spanish Contract MTM2007-67908-C02-02. The authors thank to the Facultad de Ciencias Matem\'aticas of the Universidad Complutense de Madrid for excellent working conditions.}
\date{28-1-2010}



\maketitle

\section{Introduction}

In~\cite{Mi} Milnor introduced the Milnor fibration for any holomorphic germ 
\[f:(\CC^n,O)\to\CC\]
 and proved that the Milnor fibre
is always a CW-complex of dimension at most $(n-1)$. In the case in which $f$ has an isolated singularity at the origin he also proved
that the Milnor fibre is homotopy equivalent to a bouquet of $(n-1)$-spheres. The number of spheres is equal to the Milnor number $\mu$,
which can be easily computed from the equation. If $f$ has non-isolated singularities at the origin the situation
is much more complicated. Up to now, the only general result is Kato and Matsumoto bound~\cite{KM} which asserts that the 
Milnor fibre is $(s-2)$-connected, where $s$ is the codimension of the singular locus in $\CC^n$. 
The homotopy type of the Milnor fibre of a general function germ can be very complicated.
In fact, by a recent result of the first author~\cite{FdB4}
for any local analytic set in $\CC^m$ there is a function whose Milnor fibre is 
homotopy equivalent to the complement of the set in a sufficiently small ball. The class of such spaces is very rich (contains for
example the class of complements of hyperplane and line arrangements) and there is a whole theory dedicated to its study.
Hence we may not expect to find a simple
description of the homotopy type of the Milnor fibre of a general function germ. 

It is very interesting to find classes of non-isolated hypersurface singularities for which the homotopy type of the Milnor fibre 
admits an understandable description from the equation. This paper contributes to a program in this direction.
Let $I\subset\calO_{\CC^n,O}$ be an ideal defining a $3$-dimensional i.c.i.s. $\Sigma_0$
and let $f$ be a function of finite extended codimension 
with respect to $I$ (see Section~\ref{sectionunfoldings} for a definition). Our main 
results are the following:
\begin{enumerate}
\item We prove that the Milnor fibre of $f$ is homotopic to a bouquet of spheres of different dimensions
(see Theorem~\ref{homotopiabouquet}).
\item We also compute the number of spheres appearing
in terms of the equation (see Theorem~\ref{homologia}).
\end{enumerate}
Simmilar results for the cases in which $\Sigma_0$ is of dimension $1$ and $2$ were produced by the work
of Siersma~(see \cite{Si1} and~\cite{Si2}), Zaharia~\cite{Za} and Nemethi~\cite{Nm}. 
If $\Sigma$ is a hypersurface the result was proved by Shubladze~\cite{Sh} and Nemethi~\cite{Ne}.

Actually we formulate the following:\\

\noindent{\textbf{Conjecture.} {\it The Milnor fibre of a function of finite extended codimension with respect to an i.c.i.s. has 
the homotopy type of a bouquet of spheres.}\\

Functions of finite extended codimension with respect to an i.c.i.s. are a particular case of $I$-isolated singularities as defined
and studied in~\cite{FdB2}. There it was given a bouquet theorem decomposing homotopically the Milnor fibre in a bouquet of
several $(n-1)$-spheres and an unknown space (Theorem 9.3~of~\cite{FdB2}). 
The results of this paper identify the homotopy type of that space. 
It would be interesting to generalise this paper to other $I$-isolated singularities.

Other bouquet theorems in the context of singular ambient spaces were proved by Siersma~\cite{Si4} and Tibar~\cite{Ti}. 

Let us end with a description of some aplications of this kind of results. 
The class of singularities studied in this paper shows very surprising phenomena from the equisingularity viewpoint.
It has been used in~\cite{FdB2} in order to disprove several old equisingularity questions. At the 
moment of writing the paper~\cite{FdB2} some of the Betti number formula contained in this paper were known to
the first author. It was this knowledge which lead him to guess the counterexamples contained in~\cite{FdB2}
(see Section~\ref{sectionexamples} for more detais).
We hope that a systematic solution to our conjecture would lead to interesting examples showing other topological phenomena 
in non-isolated singularities as yet unknown to us.

The structure of this paper is inspired in the classical Picard-Lefschetz theory of isolated
singularities and Sierma's generalisation for non-isolated singularities. In this theory, a function is perturbed to split a singular point into
several Morse-type singularities (this process is usually refered to as Morsification). Then it is shown that the homology of the Milnor fibre of the original function can be recovered
from the Milnor fibres of each Morse-type singularity. Finally, these homologies are computed
by a local study of the Morse-type singularities. In Section~\ref{sectionunfoldings} we use the results of~\cite{FdB1}
to prove that in our case we can do a process analogous to the Morsification, but instead of obtaining
only Morse-type singularities, we will also obtain a non-isolated singularity over the Milnor
fibre of the i.c.i.s. $\Sigma$. In Section~\ref{sechomsplit} we show that, as in the isolated case, the homology of the original Milnor fibre can be recovered from the pieces of the Milnor 
fibres contained in small neighbourhoods of the singularities obtained after the deformation. Having done that, the hardest part 
of the argument is to study the Milnor fibre around the non-isolated singularity obtained
after the deformation. This study is done by taking a suitable decomposition of the Milnor fibre of $\Sigma$ in such a way that the space we want to study fibres naturally over each
stratum of this decomposition. Section~\ref{sectiondecomp} describes this decomposition, and the following ones study the parts that fibre over the different strata.
Sections~\ref{seccalx} and \ref{seccalm} show how to glue these pieces to obtain the homology of the Milnor fibre around the deformation of $\Sigma$. Finally, we
use all these data to recover the homotopy type of the whole Milnor fibre in Sections~\ref{sechomolfibramilnor} and \ref{sechomotfibramilnor}. 
The last section describes a distinguished family of functions belonging to the class studied in this paper which already had striking aplications in topological equisingularity.
\subsection{Terminology} 
If $X$ is a subspace of a topological space $Y$ we denote by $\dot{X}$ the interior points of $X$, and by $\partial X$ the boundary
points of $X$ in $Y$. 
Given two topological spaces $X$ any $Y$ we denote that they have the same homotopy type by $X\simeq Y$.
We will denote by $D_\delta$ the closed disc of radius $\delta$ in the complex plane and by $B_\epsilon$ the closed ball of radius
$\epsilon$ in a complex affine space. The centers of the discs and balls will be clear from the context unless they are explicitly 
mentioned in the text or in the notation (by $B(x,\epsilon)$). Denote by $\SSS^k$ an sphere of dimension $k$.

\section{Unfoldings}
\label{sectionunfoldings}

Let $I:=(g_1,...,g_{n-3})$ define a $3$-dimensional i.c.i.s. $\Sigma_0$ in $\CC^n$. Denote by $\Theta_{I,e}$ the germs of vector fields tangent to the i.c.i.s. A function
$f:\CC^n\to\CC$ is singular at $\Sigma$ if and only if it belongs to $I^2$. As in~\cite{Pe} we define the {\em extended codimension of $f$ with respect to $I$} as
\[c_{I,e}:=dim_\CC (I^2/\Theta_{I,e}(f)).\]
From the deformation viewpoint, functions with finite extended codimension play the same role in the space of functions singular in $\Sigma_0$ than isolated singularities in
the space of all funtion-germs. A geometric characterisation of germs of finite extended codimension was given in~\cite{Za} (see~\cite{FdB2} for another proof and
generalisations): these are germs in $I^2$ which outside the origin only have either isolated $A_1$ singularities or singularities of type $D(3,p)$, with $p\in\{0,1,2\}$.

The singularity $D(k,p)$ has the following normal form (see~\cite{Pe}):
\[\sum_{1\leq i\leq j\leq p}x_{i,j}y_iy_j+\sum_{p+1\leq i\leq n-k}y_i^2=1,\]
where $\{x_{i,j}\}_{1\leq i\leq j\leq p}\cup\{y_i\}_{1\leq i\leq n-k}$ is an independent system of linear forms in $\CC^n$.

Given a germ $f\in I^2$ we can express it as a matrix product
\[f=(g_1,...,g_{n-3})(h_{i,j})(g_1,...,g_{n-3})^t\]
with $(h_{i,j})$ a symmetric matrix of holomorphic germs of size $n-3$. An easy computation shows that the restriction $(h_{i,j})|_{\Sigma_0}$ only depends on $f$. 

Let 
\[G_1,...,G_{n-3}:\CC^n\times B\to\CC^{n-3}\]
be the semiuniversal unfolding of the i.c.i.s. $(g_1,...,g_{n-3})$. Its base $B$ is a germ of complex manifold~\cite{Lo}.
Given any $b\in B$ denote by 
\[(G_{1,b},...,G_{n-3,b}):\CC^n\to\CC^{n-3}\]
the mapping corresponding to the parameter value $b$.
In the space $SM(n-3)$ of symmetric matrices with complex entries we consider the stratification
\[SM(n-3)=\bigcup_{i=0}^{n-3}SM(n-3,i),\]
where $SM(n-3,i)$ is the set of matrices of corank equal to $i$. Notice that $\overline{SM(n-3,i)}$ consists of the set of matrices defined by the vanishing of the minors
of size $n-2-i$. It is easy to check that $SM(n-3,i)$ is of codimension $i(i+1)/2$ in $SM(n-3)$. We consider the unfolding 
\[F:\CC^n\times B\times SM(n-3)\to\CC\]
of the function $f$ defined by:
\begin{equation}
\label{unfolding1}
F(x_1,...,x_n,b,(c_{i,j})):=(G_{1,b},...,G_{n-3,b})(h_{i,j}+c_{i,j})(G_{1,b},...,G_{n-3,b})^t.
\end{equation}

\begin{notation}
\label{notaciones}
Denote by $S=B\times SM(n-3)$ the base of the unfolding. Consider $\Sigma:=V(G_1,...,G_{n-3})\subset \CC^n\times S$. 
Given any $s=(b,(c_{i,j}))\in S$ we denote by $f_s:\CC^n\to\CC$ the function corresponding to the parameter value $s$,
by $\Sigma_s$ the locus $V(G_{1,b},...,G_{n-3,b})$ and by 
\[H(F):\Sigma\to SM(n-3)\]
the mapping defined by $H(F)(x,b,(c_{i,j})):=(h_{i.j}(x)+c_{i,j})$. Consider
\[H(f_s):=H(F)|_{\Sigma_s}.\]
Define $\Sigma[i]:=H(F)^{-1}(SM(n-3,i))$ and notice that we have the obvious equality $\Sigma[i]_s=H(f_s)^{-1}(SM(n-3,i))$.

Figure~\ref{figura1} shows a schematic view of these sets.
\end{notation}

\begin{figure}[h]
\setlength{\unitlength}{0.000437445in}
\begin{picture}(8061,4753)(0,-10)
\put(6078,4522){$\Sigma_s$}
\put(723,4567){$\Sigma_0$}
\put(1803,1057){$\overline{\Sigma[1]_0}$}
\put(7293,832){$\overline{\Sigma[1]_s}$}
\includegraphics[scale=0.5]{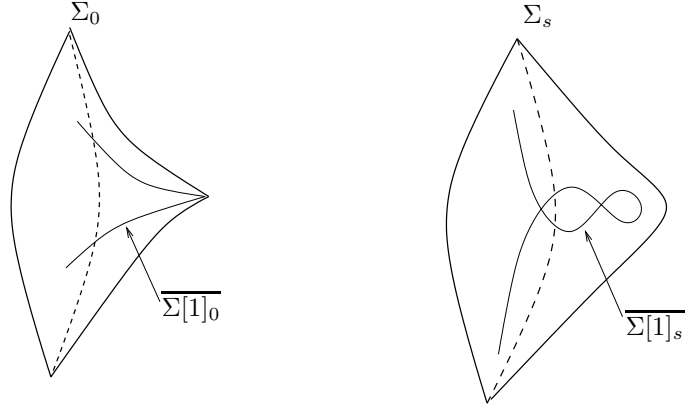}
\end{picture}
\caption{\label{figura1}The deformation of the i.c.i.s. and the stratification}
\end{figure}

The function $F_0$ coincides with $f$, where $0\in S$ is the origin of the base of the unfolding. 

Let $\epsilon$ and $\delta$ be radii for a Milnor fibration of $f$, that is radii such that
\begin{enumerate}
\item the central fibre $f^{-1}(0)$ meets $\partial B_{\epsilon'}$ transversely in the stratified sense for any $\epsilon'\leq \epsilon$,
\item for any $t\in D_{\delta}\setminus\{0\}$, the fibre $f^{-1}(t)$ meets $\partial B_\epsilon$ transversely,
\item the only critical value of $f|_{B_\epsilon}$ is $0$.
\end{enumerate}

From~\cite{FdB1} and~\cite{FdB2} we obtain:
\begin{theo}
\label{teounfolding1}
There exists a proper closed analytic subset $\Delta$ of $S$, and a ball $B_\eta$ centred at $0\in S$ such that for any $s\in B_\eta\cap (D\setminus\Delta)$ we have
\begin{enumerate}
\item for any $t\in D_\delta$ the intersection of $f_s^{-1}(t)$ with $\partial B_\epsilon$ is transversal (in the stratified sense if $t=0$).
\item the critical set of the function $f_s|_{B_\epsilon}$ is the union of $\Sigma_s\cap B_\epsilon$ with a finite number of Morse type singularities, whose critical values are 
pairwise different and different from $0$.
\item the set $\Sigma_s\cap B_\epsilon$ is smooth (a Milnor fibre of the i.c.i.s. $(\Sigma_0,O)$) and the mapping 
\begin{equation}
\label{hessiana}
H(f_s)|_{\Sigma_s\cap B_\epsilon}:\Sigma_s\cap B_\epsilon\to SM(n-3)
\end{equation}
is transversal to the stratification of $SM(n-3)$ by corank. In particular $\Sigma[i]_s$ is a manifold of codimension $i(i+1)/2$ in the $3$-dimensional 
manifold $\Sigma_s\cap B_\epsilon$. Therefore the critical points of $f_s$ in $\Sigma_s$ are of type $D(3,0)$, $D(3,1)$ or $D(3,2)$.
\end{enumerate}
Denote by $\calC$ and $\calD$ the critical set and the discriminant of the mapping
\begin{equation}
\label{fibraciongrande}
(F,pr_2):\CC^n\times B_\eta:\to\CC\times B_\eta,
\end{equation}
where $pr_2$ denotes the projection of $\CC^n\times B_\eta$ to the second factor. Then the restriction
\[(F,pr_2):(B_\epsilon\times B_\eta)\cap (F,pr_2)^{-1}((D_\delta\times B_\eta)\setminus\calD)\to (D_\delta\times B_\eta)\setminus\calD\]
is a locally trivial fibration with fibre diffeomorphic to the Milnor fibre of $f$.
\end{theo}

\begin{lema}
\label{2icis}
The set $\overline{\Sigma[1]_0}=V(det(H(f)),g_1,...,g_{n-3})$ is a $2$-dimensional i.c.i.s.
\end{lema}
\begin{proof}
For any $s\in S$ the set of points of $\Sigma_s\cap B_\epsilon$ where $H(f_s)$ has corank at least $1$ coincides with $V(det(H(f_s))$. Denote by $c_{I_x,e}((f_s)_x)$
 the extended codimension of the germ $f_s$ at $x$ with respect to the ideal $I_x$ defining the germ $(\Sigma_s,x)$. In~\cite{FdB1} it is shown
that the set of points such that $H(f_s)$ is not transversal to the stratification of $SM(n-3)$ by corank coincides precisely with the set of points at which 
$c_{I_x,e}((f_s)_x)$ is non-zero. Therefore the only point at which $H(f)$ is not transversal to the corank stratification
is the origin if we take $\epsilon$ small enough. Thus at any $x\in\Sigma_0\setminus\{0\}$ the germ $f_x$ is of type $D(3,0)$ if $det(H(f)(x)\neq 0$ and of type
$D(3,1)$ if $det(H(f)(x)=0$. Inspecting the normal form of the $D(3,1)$ singularity we find that $V(det(H(f)),g_1,...,g_{n-3})$ has an isolated singularity at
the origin. 
\end{proof}

Now we study the deformation $\overline{\Sigma[1]_s}:=V(det(H(f_s)),G_{1,s},...,G_{n-3,s})\cap B_\epsilon$ as we move in $S$. Choose $s\in S\setminus\Delta$. 
Since $H(f_s)$ is transversal to the stratification by corank there is a finite set of points $\Sigma[2]_s$ in $\Sigma[1]_s$ of type $D(3,2)$ and the rest of the points 
are of type $D(3,1)$. The normal form of the $D(3,2)$ singularity gives that $\overline{\Sigma[1]_s}$ has an $A_1$-type singularity at any point in $\Sigma[2]_s$. Let $a$ denote the 
cardinality of $\Sigma_s[2]$. In the next Lemma we show that $a$ is independent on $s$.

\begin{lema}
\label{connectedcover}
The restriction 
$$pr_2:(\Sigma,\overline{\Sigma[1]},\Sigma[2])\cap pr_2^{-1}(S\setminus \Delta)\to S\setminus\Delta$$
is a topological locally trivial fibration of triples such that $\Sigma_s$ is the Milnor fibre of the i.c.i.s. $\Sigma_0$, the surface $\overline{\Sigma[1]}_s$ 
is a deformation of 
$\overline{\Sigma[1]_0}$ having precisely $a$ singularities of $A_1$-type in $\Sigma[2]_s$. Moreover the restriction
\begin{equation}
\label{monodromy1}
 pr_2:\Sigma[2]\cap pr_2^{-1}(S\setminus \Delta)\to S\setminus\Delta
\end{equation}
is a covering and $\Sigma[2]\cap pr_2^{-1}(S\setminus \Delta)$ is connected.
\end{lema}
\begin{proof}
The topological triviality statements are easy after the normal forms of the $D(3,p)$ singularities.

The space $\Sigma$ is smooth since it is the product of the total space $ \calV$ of the versal deformation of the i.c.i.s. $V(g_1,...,g_{n-3})$ (which is smooth) with the space
$SM(n-3)$. For any matrix $M\in SM(n-3)$ the fibre $H(F)^{-1}(M)$ is diffeomorphic to $ \calV$, being the diffeomorphism $\phi_M: \calV\to H(F)^{-1}(M)$ defined by
$\phi_M(x):=(x,M-(h_{i,j}(x))$. This shows that $H(F):\Sigma\to SM(n-3)$ is a trivial fibration. The set $\Sigma[2]$ is connected because it is a Zariski dense open subset in
the analytic manifold $H(F)^{-1}(SM(n-3,2))$, which is diffeomorphic to the product $ \calV\times SM(n-3,2)$, being $SM(n-3,2)$ irreducible.
\end{proof}

We summarise for further reference the main invariants introduced for the function $f$.

\begin{definition}
\label{defnueva}
Define $\mu_0$ and $\mu_1$ to be the Milnor numbers at the origin of the i.c.i.s. $\Sigma_0$ and $\overline{\Sigma[1]_0}$. For $s\in S\setminus\Delta$ close to the origin 
we define $a$ to be the cardinality of $\Sigma[2]_s$ and $\#A_1$ to be the number of Morse points of $f_s$.
\end{definition}

\subsection{The $corank=2$ case}
\label{corrango2}

We will need an slightly larger unfolding of $f$ in the particular case in which $corank(H(f)(O))$ is precisely equal to $2$. In that case, after possibly changing the 
generators of the i.c.i.s we can assume that $f$ is of the form
\begin{equation}
\label{formaespecial}
f=(g_1,g_2)(h_{i,j})(g_1,g_2)^t+\sum_{i=3}^{n-3}g_i^2.
\end{equation}
When $c_{I,e}(f)$ is finite the mapping 
\[H(f):\Sigma\to SM(2)\]
is transverse to the corank stratification outside the origin. Therefore the origin is an isolated point of the locus where $corank(H(f))$ is at least $2$. 
Since in this case the corank $2$ locus is defined 
by the vanishing of the $n$ functions $h_{1,1},h_{1,2},h_{2,2},g_1,...,g_{n-3}$, we have that $\Sigma[2]_0$ is a $0$-dimensional i.c.i.s. concentrated at the origin. Let $a$ be 
its length as $0$-dimensional scheme. Let 
\[\Sigma[2]=V(H_{1,1},H_{1,2},H_{2,2},G_1,...,G_{n-3})\subset\CC^n\times S\to S\]
be the versal deformation of the i.c.i.s. $V(h_{1,1},h_{1,2},h_{2,2},g_1,...,g_{n-3})$. Any fibre $\Sigma[2]_s$ is a $0$-dimensional scheme of lenght $a$. The discriminant
 $\Lambda\subset S$ is the set of parameters where $\Sigma[2]_s$ is reduced. By versality the discriminant is irreducible and
 reduced~(Corollary~4.11~and~Proposition~6.11~of~\cite{Lo}), and its 
smooth locus $\dot{\Lambda}$ is the set of parameters such that $\Sigma[2]_s$ has exactly a fat point of length $2$ and is otherwise reduced 
(Lemma~4.9~of~\cite{Lo}).

\begin{definition}
\label{vanishcycle}
Fix a base point $s_0\in S\setminus\Lambda$. Any path $\gamma:[0,1]\to S$ such that $\gamma(0)=s_0$, $\gamma([0,1))$ is included in $S\setminus\Lambda$ and $\gamma(1)$ belongs
to a smooth point of $\Lambda$ induces a deformation $\{\Sigma[2]_t\}_{t\in [0,1]}$ along $\gamma$ such that precisely two points $\{p_0,p_1\}$ in $\Sigma[2]_0=\Sigma[2]_{s_0}$ collapse to the same
point in $\Sigma[2]_1$. The {\em vanishing cycle in $\Sigma[2]_{s_0}$ associated to $\gamma$} is, by definition, the pair $\{p_0,p_1\}$.
\end{definition}

\begin{lema}
\label{fullequivalence}
All the points of $\Sigma[2]_{s_0}$ are at the same equivalence class by the equivalence relation generated by the vanishing cycles.
\end{lema}
\begin{proof}
The base $S$ of the versal unfolding 
\[\psi:\Sigma[2]\to S\]
of the $0$-dimensional i.c.i.s. $\Sigma[2]_0$ can be identified with a neighbourhood $U$ of the origin in $\CC^N$. We choose a straight line $l$ through
$s_0$ such that $l$ meets $\Lambda$ transversely at smooth points. The neighbourhood and the line can be chosen so that $\psi^{-1}(l\cap U)$ is the Milnor fibre of a 
$1$-dimensional i.c.i.s. Therefore $\psi^{-1}(l\cap U)$ is connected (\cite{Lo}~Chapter~5). Choose a system of paths $\{\gamma_i\}_{i=1}^M$ joining $s_0$ with each of the
points of $l\cap\Lambda$, without self intersections and not intersecting pairwise except at $s_0$. Since the space $\psi^{-1}(l\cap U)$ is homotopy equivalent to the result 
of attaching a $1$-cell at each of the vanishing cycles associted to the paths $\{\gamma_i\}_{i=1}^M$, the connectivity of $\psi^{-1}(l\cap U)$ proves the lemma.
\end{proof}

Given any finite set $K$ denote by  $Aut(K)$ its permutation group. The monodromy action
\[\rho:\pi_1(S\setminus\Lambda,s_0)\to Aut(\Sigma[2]_{s_0})\]
induces a monodromy action 
\[\rho_2:\pi_1(S\setminus\Lambda,s_0)\to Aut(Sym^2(\Sigma[2]_{s_0})\setminus Diag,\]
where $Sym^2(\Sigma[2]_{s_0})\setminus Diag$ is the second symmetric product of $\Sigma[2]_{s_0}$ minus its diagonal (that is the set of subsets of cardinality precisely $2$). The set 
of vanishing cycles is a subset of $Sym^2(\Sigma[2]_{s_0})\setminus Diag$.

\begin{lema}
\label{monodromiatransitiva}
The monodromy action $\rho_2$ preserves the set of vanishing cycles and acts transitively on it. 
In other words, the set of vanishing cycles is an orbit by the monodromy action.
\end{lema}
\begin{proof}
A vanishing cycle induced by a path $\gamma$ is transformed by an element $[\alpha]\in\pi_1(S\setminus\Lambda,s_0)$ to the vanishing cycles induced by the concatenation 
path $\alpha\centerdot\gamma$. The transitivity is a classical consequence of the irreducibility of the discriminant~(\cite{AGV}~Chapter~3).
\end{proof}

We enlarge the unfolding of $f$ defined in~(\ref{unfolding1}) by considering the following one instead of it:
\[F:\CC^n\times S\to\CC\]
given by
\begin{equation}
\label{unfoldinggrande}
F:=(G_1,G_2)(H_{i,j})(G_1,G_2)^t+\sum_{i=3}^{n-3}G_i^2.
\end{equation}

The statements of Theorem~\ref{teounfolding1} and Lemma~\ref{2icis} clearly remain true for this unfolding. 

\section{Homology splitting}
\label{sechomsplit}

We will compute the homology of the Milnor fibre using a general method of Siersma~\cite{Si1},~\cite{Si2},~\cite{Si3}.

We have chosen radii $\epsilon$ and $\delta$ for a Milnor fibration of $f$. In that situation the total space $X_0:=B_\epsilon\cap f^{-1}(D_\delta)$ of the representative 
\[f:X_0\to\CC\]
is contractible.

Consider the versal unfolding $F:\CC^n\times S\to\CC$ defined in the previous section.
Choose a direction in $S$ such that the line $l$ through the origin $O$ of $S$ in this direction has $O$ as an isolated point of $\Delta$. Let $D_\xi$ be a disc in 
$l$ around $O$ only meeting $\Delta$ in $O$. Consider the associated $1$ parameter unfolding
\[F:\CC^n\times D_\xi\to\CC.\]
Denote by $f_s$ the function $f_s(x):=F(x,s)$. By Ehresmann fibration theorem $X_s:=B_\epsilon\cap f_s^{-1}(D_\eta)$ is diffeomorphic to $X_0$ and hence it is contractible. If $s\neq 0$ the function
\[f_s:X_s\to D_\delta\]
is a locally trivial fibration over $D_\delta\setminus\{0,v_1,...,v_r\}$, where $\{0,v_1,...,v_r\}$ are the critical values of $f_s$. Each $v_i\neq 0$ is the image of 
precisely $1$ singular point of type $A_1$ of $f_s$. The fibre of $f_s$ over any point $w$ not in $\{0,v_1,...,v_r\}$ is diffeomorphic to the Milnor fibre of $f$. 
Therefore we are interested in the homology $H_k(f_s^{-1}(w);\ZZ)$, which is isomorphic to $H_{k+1}(X_s,f_s^{-1}(w);\ZZ)$ by the contractibility of $X_s$.

Consider $D_0,D_1,\ldots,D_{r}$ a system of disjoint small disks inside $D_\delta$ centered in $0,v_1,\ldots v_{k}$ respectively. Choose points $t_i\in \partial D_i$, 
and disjoint paths $\alpha_i$ joining $t_0$ with $t_i$. We can take $w=t_0$. Define
\[G:=\cup_{i=1}^{r}\alpha_i \bigcup \cup_{i=0}^{r}D_i.\]
It is clear that $G$ is a deformation retract of $D_\Delta$, and since $f_s$ is a locally trivial fibration outside $G$, we have that 
\[H_k(X_s,f_s^{-1}(w))=H_k(f_s^{-1}(G),f_s^{-1}(w)).\]
By excision: 
\[H_k(f_t^{-1}(G),f_s^{-1}(w))=\oplus_{i=0}^{r}H_k(f_s^{-1}(D_i),f_s^{-1}(t_i)).\]

It is classical from Picard-Lefschetz theory that for any $i>0$ we have
\[H_n(f_s^{-1}(D_i),f_s^{-1}(t_i))\cong\ZZ\] and 
\[H_k(f_s^{-1}(D_i),f_s^{-1}(t_i))=0\]
if $k\neq n$.

Now let $T$ be a tubular neighbourhood of $\Sigma_s$. We can take for example $T:=(G_{1,s},...,G_{n-3,s})^{-1}(B)$ for $B$ a small ball around the origin in $\CC^{n-3}$. Taking $T$ small enough and $D_0$ small in comparison with $T$ we can assume that $f_s^{-1}(t)$ meets the boundary $\partial T$ transversely for any $t\in D_0$. 
We define
\begin{equation}
\label{partechunga}
\calM:=f_s^{-1}(t_0)\cap T.
\end{equation}

By this tranversality, and because of the fact that $f_s$ has no critical points in $f_s^{-1}(D_0)\setminus T$, the part of the space $f_s^{-1}(D_0)$ that lives outside
$T$ can be retracted to $f_s^{-1}(t_0)$. This means that the pair $(f_s^{-1}(D_0),f_s^{-1}(t_0))$ is homotopy equivalent to the pair $(f_s^{-1}(t_0)\cup T,f_s^{-1}(t_0))$. By excision we have 
\[H_k(f_s^{-1}(D_0),f_s^{-1}(t_0))\cong H_k(T,f_s^{-1}(t_0)\cap T)=H_k(T,\calM).\]
We summarise what we have obtained:

\begin{prop}
\label{homsplit}
Denote the Milnor fibre of $f$ by $\mathbf{F}_f$. 
Let $r$ be the number of $A_1$ points that $f_s$ has in $B_\epsilon$ for $s\in S\setminus\Delta$ close to the origin of $S$. Then 
\[H_{n-1}(\mathbf{F}_f;\ZZ)\cong H_n(T,\calM;\ZZ)\oplus\ZZ^r,\]
\[H_{k}(\mathbf{F}_f;\ZZ)\cong H_{k+1}(T,\calM;\ZZ)\]
for $1\leq k\leq n-2$. By connectivity
\[H_{0}(\mathbf{F}_f;\ZZ)\cong\ZZ.\]
\end{prop}

By construction, $T$ is homotopic to $\Sigma_s$, which is the Milnor fibre of $\Sigma_0$. We will spend a large part of this paper computing the homology of $\calM$. 

\section{The Milnor fibre of the $D(k,p)$ singularity}

We collect and reprove the following proposition, which follows from~\cite{FdB3}~and~\cite{Ga}.
\label{secdkp}
\begin{prop}
\label{milnordkp}
The Milnor fibre of the $D(k,p)$ singularity in $\CC^n$ is homotopy-equivalent to the sphere $\SSS^{n+p-k-1}$.
\end{prop}
\begin{proof}
Since the $D(k,p)$ in $\CC^n$ singularity is quasi-homogeneous its Milnor fibre is diffeomophic to the global hypersurface $X\subset\CC^n$ defined by
\[\sum_{1\leq i\leq j\leq p}x_{i,j}y_iy_j+\sum_{p+1\leq i\leq n-k}y_i^2=1,\]
where $\{x_{i,j}\}_{1\leq i\leq j\leq p}\cup\{y_i\}_{1\leq i\leq n-k}$ is an independent system of linear forms in $\CC^n$.
Hence $X$ is homotopic to the $n-p-k$-suspension of the hypersurface $Y\subset\CC^{\frac{p^2+3p}{2}}$ defined by
\[\sum_{1\leq i\leq j\leq p}x_{i,j}y_iy_j=1.\]
The projection 
\[\sigma:Y\to\CC^k\setminus\{O\}\]
defined by $\sigma(x_{i,j},y_l):=y_l$ is a locally trivial fibration with fibre an affine hyperplane in $\CC^{\frac{p^2+p}{2}}$, and hence contractible.
We conclude that $Y$ is homotopy equivalent to the unit sphere $\SSS^{2k-1}$ in $\CC^k$. Consequently $X$ is homotopic to the sphere $\SSS^{n+k-p-1}$.
\end{proof}

\begin{lema}
\label{ecuador}
Given the $D(1,1)$ singularity
\[f:=x_1y_2^2+y_3^2...+y_n^2:(\CC^n,O)\to\CC\]
in $\CC^n$ its restriction
\[f|_{H_1}:H_1\to\CC\] 
to the hyperplane $H_1$ defined by $x_1=1$ is a Morse singularity at the origin. The pair of Milnor fibres $(f^{-1}(t),(f|_{H_1})^{-1}(t))$ is homotopy equivalent 
to the pair $(\SSS^{n-1},\SSS^{n-2})$ with $\SSS^{n-2}$ embedded in $\SSS^{n-1}$ as the equator.
\end{lema}
\begin{proof}
By suspension it is enough to consider $n=2$, and in this case it is obvious, since the Milnor fibre of $x_1x_2^2$ projects to the double cover of $\CC\setminus\{O\}$
by the projection $pr(x,y)=x$ and the Milnor fibre of $f|_{H_1}$ is a fibre of this projection.
\end{proof}

Consider the $D(3,2)$ singularity
\[f:=x_1y_1^2+2x_2y_1y_2+x_3y_2^2+y_3^2+...+y_{n-3}^2:(\CC^n,O)\to\CC\]
in $\CC^n$.

Recall from Notation~\ref{notaciones} the definition of $\Sigma[i]$. As here we are considering no unfolding we have the equality $\Sigma=\Sigma_0$.
The restriction of $f$ to any $(n-3)$-dimensional transversal to $\Sigma[0]$ is a Morse type singularity in $\CC^{n-3}$.
The restriction of $f$ to any $(n-2)$-dimensional transversal to $\Sigma[1]$ is a $D(1,1)$ singularity in $\CC^{n-2}$.

The stratum $\Sigma[1]$ is equal to $V(det(H(f))\setminus\{0\}$ and hence is homotopic to its link $L_\epsilon:=\Sigma[1]\cap\SSS_\epsilon$ at the origin. 
Since the singularity is homogeneous we can take $\epsilon=1$ and denote $L_\epsilon$ by $L$. This link is diffeomorphic to 
$\RR\PP^3$, since the surface $\Sigma[1]$ is defined by $det(H(f))(x_1,x_2,x_3)=x_1x_3-x_2^2=0$ in $\Sigma\cong\CC^3$. 
Hence its fundamental group is isomorphic to $\ZZ_2$.

Let 
\[\kappa:N(\CC^n,\Sigma)\to \Sigma\]
\[\kappa_1:N(\CC^n,\Sigma[1])\to\Sigma[1]\]
be the holomorphic normal bundles of $\Sigma$ and $\Sigma[1]$ in $\CC^n$ respectively. We have the inclussion of restrictions
\begin{equation}
\label{normalbund}
\iota:N(\CC^n,\Sigma)|_{\Sigma[1]}\hookrightarrow N(\CC^n,\Sigma[1]),
\end{equation}
compatible with the bundle maps.

Observe that, since in this case $\Sigma$ is a $3$-dimensional coordinate subspace, the first bundle is 
trivial with fibre $\CC^{n-3}$. Notice that, as $L$ is compact,
there is a positive $\rho$ such that the $\rho$-neighbourhood of the zero section of the restriction 
\[\kappa_1|_{L}:N(\CC^n,\Sigma[1])_{L}\to L\]
embedds in $\CC^n$ holomorphically on each fibre. We denote by 
\[\kappa^{\rho}_1:N^\rho(\CC^n,\Sigma[1])|_{L}\to L\]
the embedded $\rho$-neighbourhood; its fibre is a complex $(n-2)$-dimensional ball.

For any $y\in L$ the restriction of $f$ to the fibre of the embedded normal bundle
\[f|_{N^\rho(\CC^n,\Sigma[1])_y}:(N^\rho(\CC^n,\Sigma[1])_y,y)\to\CC\]
is a $D(1,1)$ singularity with critical set the line $Crit(y):=N^\rho(\CC^n,\Sigma[1])_y\cap\Sigma$. 

Since the restriction of the function $det(H(f))$ to $Crit(y)$ is non-singular at the point $y$ for any $y\in L$, 
for $u$ small enough the intersection $Crit(y)\cap det(H(f))^{-1}(u)$ is a unique point for any
$y\in L$, and hence 
\[\Xi_u:=det(H(f))^{-1}(u)\cap\Sigma\cap N^\rho(\CC^n,\Sigma[1])|_L\] 
defines a cross-section of the embedded normal bundle, and the restriction
\[\kappa_1^{\rho}|_{\Xi_u}:\Xi_u\to L\]
is a diffeomorphism.

Let 
\[\kappa^{\rho}:N^\rho(\CC^n,\Sigma)|_{\Xi_u}\to\Xi_u\]
be an holomorphic embedding of a $\rho$-neighbourhood of the zero section of the restriction to $\Xi_u$ of the normal bundle of
$\Sigma$ in $\CC^n$. It is a trivial bundle with fibre a complex $n-3$ dimensional ball. 
For any $x\in\Xi_u$ the restriction of $f$ to the fibre $(N^\rho(\CC^n,\Sigma)|_{\Xi_u})_x$ is a Morse type singularity.

If $\rho$ and $u$ are chosen small enough we may assume that, for any $x\in\Xi_u$, we have an inclussion of fibres
\[(N^\rho(\CC^n,\Sigma)|_{\Xi_u})_x\subset (N(\CC^n,\Sigma[1])_{L})_{\kappa_1^{\rho}|_{\Xi_u}(x)}.\] 

Now we study the restrictions of the bundle maps $\kappa^{\rho}$ and $\kappa^\rho_1$ to a fibre of $f^{-1}(\delta)$ for small $\delta$. Define
\[\alpha:=\kappa_1^\rho|_{f^{-1}(\delta)\cap N^\rho(\CC^n,\Sigma[1])|_{L}}:f^{-1}(\delta)\cap N^\rho(\CC^n,\Sigma[1])|_{L}\to L,\]
\[\beta:=\kappa_1^\rho|_{\Xi_u}\comp\kappa^\rho|_{f^{-1}(\delta)\cap N^\rho(\CC^n,\Sigma)|_{\Xi_u}}:f^{-1}(\delta)\cap N^\rho(\CC^n,\Sigma)|_{\Xi_u}\to L.\]
We have 

\begin{lema}
\label{monodromiaecuador}
For $\delta$ small enough the mapping
\[(\alpha,\beta):(f^{-1}(\delta)\cap N^\rho(\CC^n,\Sigma[1])|_{L},f^{-1}(\delta)\cap N^\rho(\CC^n,\Sigma)|_{\Xi_u}\to L\]
is a locally trivial fibration of pairs with fibre homotopic to 
$(\SSS^{n-3},\SSS^{n-4})$, with $\SSS^{n-4}$ embedded in $\SSS^{n-3}$ as an equator 
and whose monodromy is isotopic to the identity in $\SSS^{n-4}$ and is the reflection over the equator in $\SSS^{n-3}$.
\end{lema}
\begin{proof}
The statement of the homotopy type of the fibre is just Lemma~\ref{ecuador}. 

The circle $\gamma(\theta):=(0,0,e^{2\pi i\theta})$ parametrises a generator of the fundamental group of $L$.
The normal bundle $N^\rho(\CC^n,\Sigma[1])|_{L}$ can be choosen so that for any $\theta$ the line 
$Crit(\gamma(\theta))$ is equal to $V(x_2,x_3-e^{2\pi i\theta},y_1,...,y_{n-3})$ and 
the cross-section $\Xi_u$ is defined by $\Xi_u(\gamma(\theta))=(u e^{-2\pi i\theta},0,e^{2\pi i\theta},0...,0)$.

For any $\theta\in [0,2\pi]$ the pair of fibres $(\alpha^{-1}(\gamma(\theta)),\beta^{-1}(\gamma(\theta))$ is homotopic to the pair of varieties $(X_\theta,Y_\theta)$ 
defined by
\[X_\theta:=V(x_2,x_1y_1^2+e^{2\pi i\theta}y_2^2+y_3^2+\cdots+y_{n-3}^2),\]
\[Y_\theta:=V(x_2,ue^{-2\pi i\theta}y_1^2+e^{2\pi i\theta}y_2^2+y_3^2+\cdots+y_{n-3}^2,x_1-ue^{-2\pi i\theta}).\]
The family of diffeomorphism
\[\varphi_\theta:\CC^n\to\CC^n\]
defined by 
\[\varphi_\theta(x_1,x_2,x_3,y_1,\ldots,y_{n-3}):=(e^{-2\pi i\theta}x_1,x_2,x_3,e^{\pi i\theta}y_1,e^{-\pi i\theta}y_2,y_3,\ldots,y_{n-3})\]
induces a diffeomorphism from $(X_0,Y_0)$ to $X_\theta,Y_\theta)$ for any $\theta\in [0,2\pi]$. Therefore a geometric monodromy is given by 
\[\varphi_1:(X_0,Y_0)\to (X_1,Y_1)=(X_0,Y_0).\]

The pair $(X_0,Y_0)$ is homotopic to $(\SSS^{n-3},\SSS^{n-4})$ and it is easy to chech that $\varphi_1$ preserves the orientation in $\SSS^{n-4}$ and reverses it
in $\SSS^{n-3}$.
\end{proof}

\section{A decomposition of $\calM$}
\label{sectiondecomp}

In Proposition~\ref{homsplit} it has become clear that in order to understand the homology of the Milnor fibre of $f$ we need to compute the homology of the 
intersection $\calM=f_s^{-1}(t_0)\cap T$, with $s\in S\setminus\Delta$ small enough and $t_0\neq 0$ small enough. 
The tubular neighbourhood $T$ is the total space of a trivial fibration
\begin{equation}
\label{tubneighb}
\pi:T\to\Sigma_s
\end{equation}
with fibre a $(n-3)$-complex dimensional ball. If $B$ is a subspace of $\Sigma_s$ we denote $\pi^{-1}(B)$ by $T_B$.

By Theorem~\ref{teounfolding1}, for a generic parameter $s$ close to the origin of the base $S$ of the unfolding of $f$, the maximal corank of 
$H(f_s)(x)$ is two for any $x\in\Sigma_s$. Recall that the set of points where the corank is at least $1$ is the surface $\overline{\Sigma[1]_s}$ defined
by the vanishing of $det(H(f_s))$. 
The singular points $\Sigma[2]_s=\{p_1,\ldots,p_a\}$ of $\overline{\Sigma[1]_s}$
are of Morse type and coincide precisely with the points where the corank of $H(f_s)$ equals $2$. 

For each point $p_i$ let $B_i(\zeta)$ be a ball of radius $\zeta$ around $p_i$ in $\CC^n$ such that $f_s|_{B_i(\zeta)}$ is biholomorphic to 
the restriction of the singularity $D(3,2)$ to the unit ball of $\CC^n$. Taking $\zeta$ small enough we can assume that the balls are mutually 
disjoint and that the intersections $A_i(\zeta):=B_i(\zeta)\cap\Sigma_s$ are balls in
$\Sigma_s$ centered in each of the points $p_i$. Taking $T$, $\zeta$, $s$, and $t_0$ small enough the space
\begin{equation}
\label{calA_i}
\calA_i:=f_s^{-1}(t_0)\cap \pi^{-1}(A_i(\zeta))=\calM\cap\pi^{-1}(A_i(\zeta))
\end{equation}
is diffeomeorphic to the Milnor fibre of $f_s$ at $p_i$ for any $i$, and hence homotopy equivalent to $\SSS^{n-2}$.

We take $\zeta$ small enough so that $\partial A_i(\zeta')$ is transverse to $\overline{\Sigma[1]_s}$ for any $0<\zeta'<\zeta$. 
We choose $\zeta_0<\zeta$ sufficiently close to $\zeta$ so that the inclussion
\[\calM\cap \pi^{-1}(A_i(\zeta'))\subset \calA_i\]
is a homotopy equivalence for any $\zeta_0\leq \zeta'\leq \zeta$.

Choose $\xi>0$ small enough so that $det(H(f_s))^{-1}(u)$ meets $\partial A_i(\zeta')$ transversely for any $u\in D_\xi$ and $\zeta_0\leq \zeta'\leq \zeta$. 
Define
\[B:=det(H(f_s))^{-1}(D_\xi)\setminus(\cup\dot{A}_i(\zeta_0)),\]
\[B_u:=det(H(f_s))^{-1}(u)\setminus(\cup\dot{A}_i(\zeta_0)).\]

A schematic picture of this decomposition can be seen in Figure~\ref{figura2}. 

\begin{figure}[h]
\setlength{\unitlength}{0.0006249975in}
\begin{picture}(7297,6564)
\put(6090,6168){$B_0$}
\put(-200,2568){$B$}
\put(2615,2668){$A_i(\zeta_0)$}
\put(6090,2493){$A_i(\zeta)$}
\includegraphics[scale=0.75]{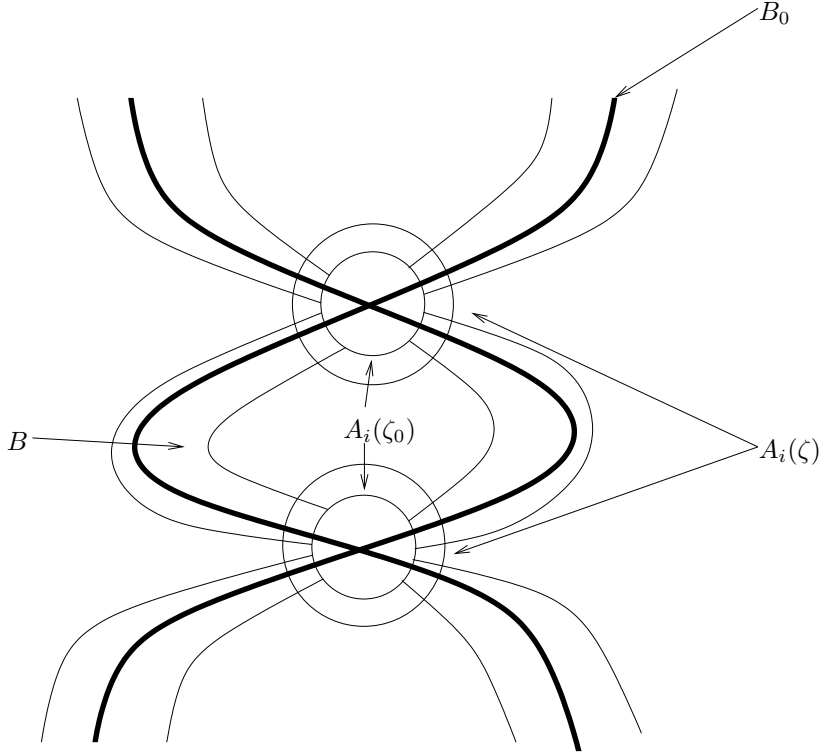}
\end{picture}
\caption{\label{figura2}The decomposition of $\Sigma_s$}
\end{figure}

The space $B$ is a tubular neighbourhood of $B_0$ in $\Sigma_s\setminus(\cup\dot{A}_i(\zeta_0))$. The mapping
\[det(H(f_s)):B\to D_\xi\]
is a trivial fibration. Therefore there is a product structure $B\cong B_0\times D_\xi$ and the projection 
\[\rho:B\to B_0\]
to the first factor induces a diffeomorphism 
\[\sigma_u:B_u\to B_0\]
for any $u$.

The restriction
\[\rho\comp\pi|_{T_B}:T_B\to B_0\] 
is a locally trivial fibration with fibre a polycylinder of complex dimension $n-2$. Define
\[\calB:=\calM\cap T_B,\]
the piece of $\calM$ falling over $B$.
Taking $T$, $\xi$ and $t_0$ sufficiently small we have that the restriction
\begin{equation}
\label{fibracsn-3}
\rho\comp\pi|_{\calB}:\calB\to B_0
\end{equation}
is a locally trivial fibration with fibre diffeomorphic to the Milnor fibre of the 
$D(1,1)$ in $\CC^{n-1}$, and hence homotopy equivalent to $\SSS^{n-3}$.

For any $\xi'>0$ we define 
\[U_{\xi'}:=\Sigma_s\setminus det(H(f_s))^{-1}(\dot{D}_{\xi'}),\]
\[\calU_{\xi'}:=\pi^{-1}(U_{\xi'})\cap\calM,\]
the complement of a tube around $\overline{\Sigma[1]_s}$ in $\Sigma_s$, and the piece of $\calM$ lying over it. 
For $T$ and $t_0$ small enough the restiction
\[\pi|_{\calU_{\xi'}}:\calU_{\xi'}\to U_{\xi'}\]
is a locally trivial fibration with fibre diffeomorphic to the Milnor fibre of the Morse singularity in $\CC^{n-3}$, and hence homotopic to $\SSS^{n-4}$.

We fix a positive $\xi_0$ smaller and close to $\xi$ and define:
\[U:=U_{\xi_0},\]
\[\calU:=\calU_{\xi_0}.\]
The restriction
\begin{equation}
\label{calU}
\pi|_{\calU}:\calU\to U
\end{equation}
is locally trivial with fibre homotopic to $\SSS^{n-4}$.

Fix a point $u$ in $\partial D_{\xi_0}$. Define
\[\calB_{u}:=\pi^{-1}(B_u)\cap\calM.\]
The mapping 
\[\rho|_{B_u}:B_u\to B_0\]
is a diffeomorphism. Hence the mapping
\begin{equation}
\label{fibpar}
((\rho|_{B_u})^{-1}\comp\rho\comp\pi|_{\calB},\pi|_{\calB_u}):(\calB,\calB_u)\to B_u
\end{equation}
is a locally trivial fibration of pairs with fibre the pair $(\SSS^{n-3},\SSS^{n-4})$, being $\SSS^{n-4}$ embedded as an equator of $\SSS^{n-3}$.

\section{The topology of $B_0$}

\subsection{The fundamental group of $B_0$}

The space $SM(k)$ of symmetric matrices of size $k$ with complex coefficients is a complex vector space of dimension $k(k+1)/2$.
The smooth locally closed algebraic subset $SM(k,l)$ has codimension
$l(l+1)/2$, and we have seen that its Zariski closure $\overline{SM(k,l)}$ is defined by the vanishing of all $l\times l$ minors.
It is easy to check that $\overline{SM(k,l)}$ is far to be, in general, a complete intersection.

Define $MM(k\times (k-1))$ to be the set of $(k\times (k-1))$ matrices of maximal rank. 
\begin{lema}
 \label{piunofacil}
The fundamental group of $MM(k\times (k-1))$ is trivial.
\end{lema}
\begin{proof}
 The space of matrices $k\times (k-1)$ which are not of maximal rank is of codimension at least $2$ in a complex vector space.
\end{proof}

The mapping
\[\alpha_k:MM(k\times (k-1))\to SM(k,1)\] 
given by 
\[\alpha_k(M):=M M^t\]
is a locally trivial fibration (by homogeneity of the action of the general linear group). Denote by $F_k$ the fibre over the matrix $A=(a_{i,j})$,
where $a_{i,j}=\delta_{ij}$ unless $i=j=k$, in which case, $a_{k,k}=0$.

\begin{lema}
\label{fkcomponents} The fiber $F_k$ has two connected components. 
\end{lema}

\begin{proof}
 We will work by induction over $k\geq 2$. For $k=2$, it is a direct computation.

Now let us compute the fibre $F_k$. Consider the following matricial equation:
\[(m_{i,j}) (m_{j,i})=(a_{i,j}).\]
Let $v_i$ be the vector in $\CC^{k-1}$ given by the $i$-th row of $(m_{i,j})$. Denote by $Re(v_i)$ and $Im(v_i)$ its real and imaginary parts respectively.

Now the previous matricial equation becomes the following system of vector equations: if $(i,j)\neq(k,k)$ then 
\[Re(v_i)\cdot Re(v_j)=\delta_{ij}+Im(v_i)\cdot Im(v_j)\]
\[Re(v_i)\cdot Im(v_j)=0\]
and
\[Re(v_k)\cdot Re(v_k)=Im(v_k)\cdot Im(v_k)\]
\[Re(v_k)\cdot Im(v_k)=0,\]
where $v\cdot w$ denotes the standard scalar product in $\RR^{k-1}$.

Consider the projection 
$MM(k\times (k-1))\subset(\CC^{k-1})^k\to\CC^{k-1}$
 to the first component. Let $B_k$ be the image of $F_k$ under this projection. 
It is easy to check that the restriction
\[\tau_k:F_k\to B_k\]
is a locally trivial fibration.

Obviously $B_k$ is 
the set of vectors $v_1$ satisfying the above system of equations for $i=j=1$. The vector $v_1$ belongs to $B_k$ if and only if $||Re(v_1)||^2$ is one unit longer than $||Im(v_1)||^2$ and both vectors are orthogonal. That is, the vector $Re(v_1)$ can be anywhere except in the interior of the unit sphere in $\RR^{k-1}$. If  $||Re(v_1)||^2$ equals $1$ then the vector $Im(v_1)$ is zero. In any other case, the vector $Im(v_1)$ lies in the $k-2$-sphere of radius $\sqrt{1-||Re(v_1)||^2}$ in the hyperplane orthogonal to $Re(v_1)$. It is easy to show that $B_k$ admits the unit sphere in $\RR^{k-1}$ embedded in the real part of $\CC^{k-1}$ as a deformation retract.

The fiber $\tau_k^{-1}((1,0,\ldots,0))$ is equal to the fiber $F_{k-1}$ of $\alpha_{k-1}$ over $A'$
where $A'$ is the result of deleting the first row and the first column in $A$.

We have constructed a fibration of $F_k$ over a space with the homotpy type of $\SSS^{k-2}$ whose fibre is $F_{k-1}$. If $F_{k-1}$ has two connected components and
$k\geq 4$, the homotopy exact sequence of the fibrations gives the result. For $k=3$ we have to check that the monodromy of the fibration does not interchange the two connected components of $F_2$, but this is direct computation.
\end{proof}

\begin{prop}
The fundamental group of $SM(k,1)$ is isomorphic to $\ZZ_2$.
\end{prop}
\begin{proof}
This is just the homotopy exact sequence of the fibration $\alpha_k$.
\end{proof}

\begin{prop}
\label{grupofund}
The fundamental group of $B_0$ is isomorphic to $\ZZ_2$.
\end{prop}
\begin{proof}
The unfolding
\[f_s=(g_1-s_1,\ldots,g_{n-3}-s_{n-3})(m_{i,j}+s_{i,j})(g_1-s_1,\ldots,g_{n-3}-s_{n-3})^t,\]
with $s\in\CC^{n-3}\times SM(n-3)$ can be obtained by pullback from the unfoldings of $f$ that we considederd in 
Section~\ref{sectionunfoldings} both in the $corank(H(f)(O)=2$ and $corank(H(f)(O)\neq 2$ cases. In both cases a generic parameter 
$s\in\CC^{n-3}\times SM(n-3)$ maps to a parameter outside the discriminant $\Delta$. Thus we can use this unfolding in order to 
compute the topology of $B_0$.

The mapping
\[\alpha:\CC^n\times SM(n-3)\to SM(n-3)\]
defined by $\alpha(x,(s_{i,j})):=(m_{i,j}(x)+s_{i,j})$ is obviously a submersion wherever it is defined. Define
\[\calZ_i:=\alpha^{-1}(\overline{SM(n-3,i)})\] 
Since $\overline{SM(n-3,i)}$ is a cone for any $i$, and the fundamental group of $SM(n-3,1)$ is isomorphic to $\ZZ_2$, we have that the local fundamental 
group of the germ $SM(n-3,1)$ at the origin is $\ZZ_2$. Since the mapping $\alpha$ is a submersion, the local fundamental group 
of $(\calZ_1\setminus \calZ_2)$ at the origin is $\ZZ_2$.

Fix a positive $\epsilon$ and a generic 
\[s^0=(s^0_1,\ldots,s^0_{n-3},(s^0_{i,j})_{1\leq i\leq j\leq n-3})\in\CC^{n-3}\times SM(n-3)\] 
sufficiently close to the origin. Consider set of functions $\{g_i\}_{i=1}^{n-3}\cup\{s_{i,j}\}_{1\leq i\leq j\leq n-3}$ in 
$\calO_{\CC^n\times SM(n-3)}$. Applying Hamm-L\^e Theorem~(Main Theorem~II.1.4~of~\cite{HL}) repeatedly for the above set of 
functions, and using the relative homotopy exact sequence we
get that the fundamental group of
\[B_\epsilon\cap (\calZ_1\setminus \calZ_2)\cap \bigcap_{i=1}^{n-3}V(g_i-s^0_i)\cap\bigcap_{1\leq i\leq j\leq n-3}V(s_{i,j}-s^0_{i,j})\]
is isomorphic to $\ZZ_2$. But it is clear that the above space is homotopic to $B_0$.
\end{proof}

\subsection{Homology of $B_0$}

We will now compute the homology of $B_0$, which is the same as $B$.
Given the function 
\[det(H(f_s)):\Sigma_s\to\CC\]
we use the Mayer-Vietoris sequence of the decomposition of 
$det(H(f_s))^{-1}(D_\xi)$ as the union of $\cup A_i(\zeta)$ and $B$ given in Section~\ref{sectiondecomp}.

The space $det(H(f_s))^{-1}(D_\xi)$ is homotopy equivalent to $det(H(f_s))^{-1}(0)$, which is homotopic to a bouquet of 
$(\mu_1-a)$ $2$-spheres (see Definition~\ref{defnueva}). 
This is because $det(H(f_s))^{-1}(0)$ is a deformation of $det(H(f))^{-1}(0)$, which is an i.c.i.s. with Milnor number $\mu_1$,
and $det(H(f_s))^{-1}(0)$ has only $a$ Morse-points 
as singularities. 

On the other hand, the intersection of each space $A_i(\zeta)$ with $B$ is the link $\RR\PP^3$ of a Morse type singularity, 
and the spaces $A_i(\zeta)$ are contractible. 

This facts, together with the computation of $\pi_1(B_0)$ allows us to compute the following Mayer-Vietoris sequence:

\begin{equation}\label{MayerVietabajo}
\xymatrix@R=-3ex@C=3ex{0\ar[r] & {\begin{array}{c} \\ \\H_3(\cup A_i(\zeta)\cap B;\ZZ) \\ \rotatebox{90}{$\cong$}\\ \ZZ^a \end{array}}  \ar[r] & {  \begin{array}{c} \\{} \\ H_3(B;\ZZ) \\ \rotatebox{90}{$\cong$}\\ \ZZ^a \end{array}  } \ar[r] & {\begin{array}{c}\\ \\ H_3(det(H(f_s))^{-1}(D_\xi);\ZZ)\\ \rotatebox{90}{$\cong$} \\ 0 \end{array}} \ar[r] & \\
\ar[r] & {\begin{array}{c} \\ \\H_2(\cup A_i(\zeta)\cap B;\ZZ)\\ \rotatebox{90}{$\cong$} \\ 0 \end{array}}  \ar[r] & {  \begin{array}{c} {}\\ \\ H_2(B;\ZZ) \\ \rotatebox{90}{$\cong$} \\ \ZZ^{\mu_1-a} \end{array}  } \ar[r] & {\begin{array}{c}\\ \\ H_2(det(H(f_s))^{-1}(D_\xi);\ZZ)\\ \rotatebox{90}{$\cong$} \\ \ZZ^{\mu_1-a} \end{array}} \ar[r]^-{\delta_2} & \\
\ar[r]^-{\delta_2} & {\begin{array}{c} \\ \\H_1(\cup A_i(\zeta)\cap B;\ZZ)\\ \rotatebox{90}{$\cong$} \\ \ZZ_2^{a} \end{array}}  \ar[r]^-{\alpha_1} & {  \begin{array}{c}\\ {} \\ H_1(B;\ZZ)\\ \rotatebox{90}{$\cong$} \\ \ZZ_2 \end{array}  } \ar[r] & {\begin{array}{c}\\ \\ H_1(det(H(f_s))^{-1}(D_\xi);\ZZ) \\   \rotatebox{90}{$\cong$} \\0 \end{array}} \ar[r] & 0\\
}
\end{equation}

\begin{remark}
\label{isorestriccion}
The restriction of the mapping $\alpha_1$ to $H_1(A_i(\zeta');\ZZ)$ is an isomorphism onto $H_1(B;\ZZ)$ for any $i$.
\end{remark}
\begin{proof}
Obvious from the proof of Proposition~\ref{grupofund}
\end{proof}

The homology of $B$ with coefficients in $\ZZ_2$ can be computed analogously, or by using the universal coefficient theorem. 
We obtain
\begin{itemize}
\item $H_4(B;\ZZ_2)=0$
 \item $H_3(B;\ZZ_2)=\ZZ_2^a$
\item $H_2(B;\ZZ_2)=\ZZ_2\oplus \ZZ_2^{\mu_1-a}$
\item $H_1(B;\ZZ_2)=\ZZ_2$
\end{itemize}

\begin{remark}
Note that the generators of $H_2(det(H(f_s))^{-1}(D_\xi);\ZZ)$ can be interpreted as follows. 
The milnor fibre of $det(H(f_0))^{-1}(0)$ has $\mu_1$ spheres as generators of its homology. Out of these spheres there are $a$ of
them which correspond to the vanishing cycles of the $a$ Morse points of $det(H(f_s))^{-1}(0)$. 
The space $det(H(f_s))^{-1}(D_\xi)$ is homotopic to $det(H(f_s))^{-1}(0)$, which in turn is homotopic to the result of collapsing 
these $a$ spheres in the Milnor fibre of $det(H(f_0))^{-1}(0)$. The remaining spheres give rise to the $\mu_1-a$ generators of 
$H_2(det(H(f_s))^{-1}(D_\xi);\ZZ)$.
\end{remark}

\section{Homology of $(\calB,\calB_u)$}
\label{homologiapar}

There are several sphere fibrations involved in the computation of the homology of the Milnor fibre, and we will need to deal with the 
corresponding Gysin sequences. These are greatly simplified if we are in the case $n\geq 8$. The homology of the Milnor fibre can be 
always deduced (by suspension) from the homology of the Milnor fibre of a function $f:(\CC^n,O)\to\CC$ with $n\geq 8$.
We will assume in $n\geq 8$ whenever is needed.

Consider the fibration $\calB_u\to B_u$. As we have seen previously, it is a fibration, with fiber homotopically equivalent to 
$\SSS^{n-4}$. This fibration can be extended to $h_s^{-1}(u)$, which is simply conected, and hence, the fibration is orientable. 
Its Gysin exact sequence leads to the isomorphisms:
\begin{equation}
\label{gysin1}
H_{i}(B_u;\ZZ)\cong H_{n-4+i}(\calB_u;\ZZ),\quad H_{i}(\calB_u;\ZZ)\cong H_{i}(B_u;\ZZ)
\end{equation}
for $i=0,1,2,3$. The rest of the homology groups of $\calB_u$ vanish.

Consider the projection $\calB\to B_u$. As we have seen before, it is a fibration with fibre homotopically equivalent to 
$\SSS^{n-3}$, and the monodromy reverses the orientation. Since the fibration is not orientable, we can only consider its 
Gysin sequence with coefficients in $\ZZ_2$, which gives the following isomorphisms:
\begin{equation}
\label{gysin2}
H_{i}(B_u;\ZZ_2)\cong H_{n-3+i}(\calB;\ZZ_2),\quad H_{i}(\calB;\ZZ_2)\cong H_{i}(B_u;\ZZ_2)
\end{equation}
for $i=0,1,2,3$. The rest of the homology groups of $\calB$ with coefficients in $\ZZ_2$ vanish.

The fibration of pairs $(\calB,\calB_u)\to B_u$ has as fibre the pair $(\SSS^{n-3},\SSS^{n-4})$ with $\SSS^{n-4}$ embedded as 
the equator of $\SSS^{n-3}$. Its monodromy acts trivally on $\SSS^{n-4}$ and reverses the hemispheres of $\SSS^{n-3}$ along the 
only non-trivial class of $\pi_1(\calB_u)\cong\ZZ^2$.

In order to compute the homology of the pair $(\calB,\calB_u)$ we can simultaneously thicken the equator $\SSS^{n-4}$ of each 
fibre to a small collar $\SSS^{n-4}\times [-\eta,\eta)]$ in $\SSS^{n-3}$. By excision we can remove fibrewise the interior of the collar.
We obtain a fibration over $B_u$ with fibre two $n-3$-disks relative to their boundary, such that the monodromy interchanges them. 

Since $\pi_1(B_u)\cong \ZZ_2$, its universal cover 
\[\sigma:\tilde B_u\to B_u\] is the only connected double cover. 
The fibration of pairs is then homologically equivalent to the composition of
an orientable fibration 
\[\varphi:\calY\to\tilde B_u\]
of $(n-3)$-spheres over $\tilde B_u$ with the covering map $\sigma$. 
The Gysin sequence of the fibration $\varphi$ gives
\[H_k(\calB,\calB_u;\ZZ)\cong H_k(\calY;\ZZ)\cong H_{k-(n-3)}(\tilde B_u;\ZZ)\]
if $k\geq n-3$ and zero otherwise.

The space $\tilde B_u$ is homotopically equivalent to the double cover $\tilde B_0$ of $B_0$ branched over its $a$ singular points, 
minus the preimage of these $a$ points. $\tilde B_0$ is a $2$-dimensional Stein space (for being a branched cover of the $2$-dimensional 
Stein space $B_0$), and hence it has the homotopy type of a $2$-dimensional CW-complex. Therefore, $H_2(\tilde B_0;\ZZ)$ is free 
and $H_3(\tilde B_0;\ZZ)$ vanishes. Since the singularities of $B_0$ are of Morse type, and the $2$-dimensional Morse singularity is the 
quotient of $\CC^2$ by the action of the group of two elements, the space $\tilde B_0$ is smooth. Hence $\tilde B_u$ is the result
of deleting from $\tilde B_0$ small balls around the $a$ preimages by the double cover of the singular points of $B_0$.
Using the Mayer-Vietoris sequence we see that such deletion leaves unchanged the homology except in dimension $3$, 
where we obtain a copy of $\ZZ$ for each deleted point. 
Sumarizing, we get that 
\begin{itemize}
\item $H_3(\tilde B_u;\ZZ)=\ZZ^a$
\item $H_2(\tilde B_u;\ZZ)=\ZZ^k$ for a certain $k$
\item $H_1(\tilde B_u;\ZZ)=0$, since it is the universal cover of $B_u$
\item $H_0(\tilde B_u;\ZZ)=\ZZ$, for it is connected.
\end{itemize}
Since the Euler characteristic of $\tilde{B}_u$ is twice the one of $B_u$, $k$ must be equal to $2\mu_1-3a+1$.

Its is easy to check that the following diagram is commutative
\[
 \xymatrix{H_{i+n-3}(\calB,\calB_u;\ZZ)\ar[r]^{\delta_{i+n-3}}\ar[d]^{\cong} & H_{i+n-4}(\calB_u;\ZZ)\ar[d]^{\cong} \\
H_i(\tilde B_u;\ZZ) \ar[r]^{\sigma_*} & H_i(B_u;\ZZ)},
\]
for any $i$, where $\delta_{i+n-3}$ is the connecting homomorphism of the long exact sequence of the pair $(\calB,\calB_u)$, 
the mapping $\sigma:\tilde B_u \to B_u$ is the covering map and the vertical arrows are the isomorphism coming from the Gysin sequences. 

Notice that the generators of $H_3(\tilde B_u;\ZZ)$ are $3$-spheres bounding balls in $\tilde{B}_0$ around the inverse image of the 
singularities of $B_0$. The generators of $H_3(B_u;\ZZ)$ are precisely the classes $[A_i(\eta)\cap B_u]$. Each of them is 
diffemorphic to $\RR\PP^3$ and doubly covered by one of the $3$-spheres. This shows that 
\[\pi_*:H_3(\tilde B_u;\ZZ)\to H_3(B_u;\ZZ),\]
and hence also $\delta_n$, is multiplication by $2$.

For being $\sigma$ a covering there is a well defined pull-back mapping in homology
\[\sigma^*:H_i(B_u;\ZZ)\to H_i(\tilde B_u;\ZZ).\]
It is clear that the map $\sigma_*\sigma^*:H_i(B_u;\ZZ)\to H_i(B_u;\ZZ)$ is multiplication by $2$ (the degree of the covering). 
This, together with the previous commutative diagram, implies that $2H_{i-1}(\calB_u;\ZZ)$ is always in the image of $\delta_i$ for
any $i$. In view of this and of the long exact sequence of the pair $(\calB,\calB_u)$ we obtain that $H_{n-2}(\calB,\ZZ)$ can not 
have $p$-torsion for $p\neq 2$.

By the above diagram and connectedness of $\tilde B_u$ we have that  $\delta_{n-3}$ is an isomorphism.

Using these facts, together with the previous computations of $H_\bullet(\calB_u;\ZZ)$, $H_\bullet(\calB;\ZZ_2)$ and 
$H_\bullet(\calB,\calB_u;\ZZ)$, plus the universal coefficients theorem allows us to completely determine
the long integral homology exact sequence of the pair $(\calB,\calB_u)$. :

\[
\xymatrix@R=-4ex{0\ar[r] & {\begin{array}{c} \\  \\H_n(\calB_u;\ZZ) \\ \rotatebox{90}{$\cong$} \\ 0 \end{array}}  \ar[r] & {  \begin{array}{c} {} \\ \\ H_n(\calB;\ZZ)\\ \rotatebox{90}{$\cong$} \\ 0 \end{array}  } \ar[r] & {\begin{array}{c} \\ \\ H_n(\calB,\calB_u;\ZZ)\\ \rotatebox{90}{$\cong$} \\ \ZZ^a\end{array}} \ar[r] & \\
\ar[r] & {\begin{array}{c} \\ \\H_{n-1}(\calB_u;\ZZ) \\ \rotatebox{90}{$\cong$}\\ \ZZ^a \end{array}}  \ar[r] & {  \begin{array}{c} \\{} \\ H_{n-1}(\calB;\ZZ) \\ \rotatebox{90}{$\cong$}\\ \ZZ_2^a\oplus \ZZ^{\mu_1-2a+1} \end{array}  } \ar[r] & {\begin{array}{c}\\ \\ H_{n-1}(\calB,\calB_u;\ZZ)\\ \rotatebox{90}{$\cong$} \\ \ZZ^{2\mu_1-3a+1} \end{array}} \ar[r] & \\
\ar[r] & {\begin{array}{c} \\ \\H_{n-2}(\calB_u;\ZZ)\\ \rotatebox{90}{$\cong$} \\ \ZZ^{\mu_1-a} \end{array}}  \ar[r] & {  \begin{array}{c} {}\\ \\ H_{n-2}(\calB;\ZZ) \\ \rotatebox{90}{$\cong$} \\ 0 \end{array}  } \ar[r] & {\begin{array}{c}\\ \\ H_{n-2}(\calB,\calB_u;\ZZ)\\ \rotatebox{90}{$\cong$} \\ 0 \end{array}} \ar[r] & \\
\ar[r] & {\begin{array}{c} \\ \\H_{n-3}(\calB_u;\ZZ)\\ \rotatebox{90}{$\cong$} \\ \ZZ_2 \end{array}}  \ar[r] & {  \begin{array}{c}\\ {} \\ H_{n-3}(\calB;\ZZ)\\ \rotatebox{90}{$\cong$} \\ \ZZ_2 \end{array}  } \ar[r] & {\begin{array}{c}\\ \\ H_{n-3}(\calB,\calB_u;\ZZ) \\   \rotatebox{90}{$\cong$} \\ \ZZ \end{array}} \ar[r] & \\
\ar[r] & {\begin{array}{c} \\ \\H_{n-4}(\calB_u;\ZZ)\\ \rotatebox{90}{$\cong$} \\ \ZZ \end{array}}  \ar[r] & {  \begin{array}{c}\\ {} \\ H_{n-4}(\calB;\ZZ)\\ \rotatebox{90}{$\cong$} \\ 0 \end{array}  } \ar[r] & {\begin{array}{c}\\ \\ H_{n-4}(\calB,\calB_u;\ZZ) \\   \rotatebox{90}{$\cong$} \\0 \end{array}} \ar[r] & \\                                                                                                                                                                                                   
}
\]

The non-zero lower homology groups are isomorphic to those of $\calB_u$, which coincide with those of $B$.

\section{Homology of $\calX$}
\label{seccalx}

Let $\calX$ be the union of $\cup_{i=1}^a \calA_i$ and $\calB$.  We will now consider the Mayer-Vietoris sequence of this union with 
coefficients in $\ZZ_2$. To do so, we need to compute the groups 
$H_\bullet(\calA_i;\ZZ_2)$ and $H_\bullet(\calA_i\cap\calB;\ZZ_2)$, since $H_\bullet(\calB;\ZZ_2)$ has already been computed.

The space $\calA_i$ is the Milnor fiber of the singularity $D(3,2)$, and hence, it has the homotopy type of the sphere $\SSS^{n-2}$.

To study the homology of $\calA_i\cap\calB$, we can use the Gysin sequence of the fibration 
\[\pi:\calA_i\cap\calB\to A_i\cap B\simeq\partial (A_i\cap det(H(f_s))^{-1}(0))\cong \RR\PP^3,\] 
with fibre $\SSS^{n-3}$. 
The groups $H_i(\RR\PP^3;\ZZ_2)$ are $\ZZ_2$ for $i=0,1,2,3$, and zero otherwise.
We obtain that $H_i(\calA_i\cap\calB,\ZZ_2)=\ZZ_2$ for $i=0,1,2,3,n-3,n-2,n-1,n$, and zero otherwise.

To study the maps $\iota_k:\bigoplus_i H_k(\calA_i\cap\calB;\ZZ_2)\to H_k(\calB;\ZZ_2)$ induced by inclussion, 
we will see them as the Gysin lift of the maps 
$\bigoplus_i H_j(A_i\cap B;\ZZ_2)\to H_j(B;\ZZ_2)$ for $j=k$ or $j=k-n+3$. Using the version of the Mayer Vietoris 
sequence~(\ref{MayerVietabajo}) with coefficients in $\ZZ_2$, we get easily
\begin{itemize}
\item $\iota_n$ and $\iota_3$ are isomorphisms.
\item $\iota_{n-1}$ is a monomorphism.
\item $\iota_{n-2}$ and $\iota_{n-3}$ are epimorphisms.
\item $\iota_1$ is an epimorphism.
\item $\iota_2$ is a monomorphism.
\end{itemize}

We need also the following Lemma, whose proof we postpone:

\begin{lema}
\label{postpuesto}
The map $\iota_2:H_{n-2}(\calA_i\cap\calB;\ZZ_2)\to H_{n-2}(\calA_i;\ZZ_2)$ induced by inclussion is an isomorphism.
\end{lema}

With all this facts, we can compute the Mayer-Vietoris sequence:

\begin{equation}\label{mayerviet-ABXz2}
\xymatrix@R=1ex@C=4ex{{\begin{array}{c} \\  \\ \oplus_i H_n(\calA_i\cap\calB;\ZZ_2) \\ \rotatebox{90}{$\cong$} \\ \ZZ_2^a \end{array}}  \ar@{^{(}->}[r] & {  \begin{array}{c@{}c@{}c} {} & & \\ & & \\  \oplus _i H_n(\calA_i;\ZZ_2) & \bigoplus  & H_n(\calB;\ZZ_2)\\ \rotatebox{90}{$\cong$} & & \rotatebox{90}{$\cong$} \\ 0 & & \ZZ_2^a\end{array}  } \ar[r] & {\begin{array}{c} \\ \\ H_n(\fraX;\ZZ_2)\\ \rotatebox{90}{$\cong$} \\ 0 \end{array}} \ar `d[l] `l[dll] [dll]+<0ex,3ex>*{} \\
{\begin{array}{c} \\ \\ \oplus_i H_{n-1}(\calA_i\cap\calB;\ZZ_2) \\ \rotatebox{90}{$\cong$}\\ \ZZ_2^a \end{array}}  \ar[r] & {  \begin{array}{c@{}c@{}c} & & \\{} & & \\ \oplus_i H_{n-1}(\calA_i;\ZZ_2) & \bigoplus & H_{n-1}(\calB;\ZZ_2) \\ \rotatebox{90}{$\cong$} & & \rotatebox{90}{$\cong$}\\ 0 & &\ZZ_2\oplus \ZZ_2^{\mu_1-a} \end{array}  } \ar[r] & {\begin{array}{c}\\ \\ H_{n-1}(\fraX;\ZZ_2)\\ \rotatebox{90}{$\cong$} \\ \ZZ_2^{\mu_1-2a+1} \end{array}} \ar `d[l] `l[dll] [dll]+<0ex,3ex>*{}\\
{\begin{array}{c} \\ \\ \oplus_i H_{n-2}(\calA_i\cap\calB;\ZZ_2)\\ \rotatebox{90}{$\cong$} \\ \ZZ_2^a \end{array}}  \ar[r] & {  \begin{array}{c@{}c@{}c} & &  {}\\ & & \\ \oplus_i H_{n-2}(\calA_i;\ZZ_2) & \bigoplus & H_{n-2}(\calB;\ZZ_2) \\ \rotatebox{90}{$\cong$} & & \rotatebox{90}{$\cong$}\\ \ZZ_2^{a} &  & \ZZ_2 \end{array}  } \ar[r] & {\begin{array}{c}\\ \\ H_{n-2}(\fraX;\ZZ_2)\\ \rotatebox{90}{$\cong$} \\ \ZZ_2\oplus\ZZ_2^{a-1} \end{array}}\ar `d[l] `l[dll] [dll]+<0ex,3ex>*{}\\
{\begin{array}{c} \\ \\ \oplus_i H_{n-3}(\calA_i\cap\calB;\ZZ_2)\\ \rotatebox{90}{$\cong$} \\ \ZZ_2^a \end{array}}  \ar[r] & {  \begin{array}{c@{}c@{}c} & &  {}\\ & & \\ \oplus_i H_{n-3}(\calA_i;\ZZ_2) & \bigoplus & H_{n-3}(\calB;\ZZ_2) \\ \rotatebox{90}{$\cong$} & & \rotatebox{90}{$\cong$}\\ 0 &  & \ZZ_2 \end{array}  } \ar@{>>}[r] & {\begin{array}{c}\\ \\ H_{n-3}(\fraX;\ZZ_2)\\ \rotatebox{90}{$\cong$} \\ 0 \end{array}}}
\end{equation}

We omit the lower part of the sequence. The non-vanishing remaining homology groups of $\calX$ are
\[H_2(\calX;\ZZ_2)\cong \ZZ_2^{\mu_1-a},\quad H_0(\calX;\ZZ_2)\cong \ZZ_2.\]

\subsection{A basis of $H_{n-2}(\calX;\ZZ_2)$}
\label{gusanosz2} 
Fix a base point $x_1\in A_1(\zeta)\cap B_u$. Choose paths $\gamma_i:[0,1]\to B_u$ such that $\gamma_1$ is a generator of the fundamental 
group of $A_1(\zeta)\cap B_u$, and $\gamma_i$ connects $x_1$ with some point $x_i\in A_i(\zeta)\cap B_u$. We choose chains $G_i \subset \calB$ 
such that the natural projection $\pi|_{G_i}$ is a locally trivial fibration over $\gamma_i$ with fibre diffeomorphic to a 
$\SSS^{n-3}$ generating the homology of the corresponding fibre of $(\rho|_{B_u})^{-1}\comp\rho\comp\pi|_\calB$. Since $\gamma_1$ is closed, the chain $G_1$ is closed with 
coefficients in $\ZZ_2$. For each $i$, we choose an $(n-2)$-sphere generating $H_{n-2}(\calA_i;\ZZ)$. 
Take a hemisphere $K_i$ of such sphere; its boundary $\partial K_i$ is an $(n-3)$-sphere in $\calA_i$. The boundary $\partial G_i$ 
consists of two $(n-3)$-spheres $L_1$ and $L_i$, being $L_i$ contained in $\calA_i$. Since $\calA_i$ is homotopic to $\SSS^{n-2}$ 
there exists a chain 
\[W_i:[0,1]\times\SSS^{n-3}\to\calA_i\]
such that $\partial W_i=\partial K_i+L_i$.

The generators of $H_{n-2}(\calX;\ZZ_2)$ are represented by the $\ZZ_2$-closed chains $Z_1:=G_1$ and $Z_i:=K_1+W_1+G_i+W_i+K_i$. 
Notice that since the coefficients are in
$\ZZ_2$ we have $K_1+W_1+C_1+W_1+K_1=G_1$, and so the way of defining the generators is homogeneous. To check that these are really 
generators we observe that $Z_2,...,Z_a$ are sent by the connecting homomorphism of the Mayer-Vietoris sequence~\ref{mayerviet-ABXz2}
to the kernel of the first mapping of the $(n-3)$-row, and that $Z_1$ generates the cokernel of the first mapping of the $(n-2)$-row. 

\begin{lema}
\label{cambiocamino}
Let $\gamma'_i:[0,1]\to B_u$ be any other path joining $x'_1$ and $x'_i$, being $x'_1$ and $x'_i$ points in $A_1(\zeta)\cap B$ and $A_i(\zeta)\cap B$
respectively. As above we can associate with $\gamma'_i$ an element $[Z'_i]\in H_{n-2}(\calX;\ZZ_2)$.
We have the equality
\[[Z'_i]=[Z_i]+m[Z_1]\]
for a certain $m\in\ZZ_2$.
\end{lema}
\begin{proof}
Let $\alpha_j$ be a path joining $x_j$ and $x'_j$ for $j=1,i$.
The product of paths $\gamma_i\centerdot\alpha_i\centerdot (\gamma'_i)^{-1}\centerdot (\alpha_1)^{-1}$ is a loop based in $x_1$. 
Since the fundamental group $\pi_1(B_u,x_1)$ is isomorphic to $\ZZ_2$ and generated by $\gamma_1$, the loop 
$\gamma_i\centerdot\alpha_i\centerdot (\gamma'_i)^{-1}\centerdot (\alpha_1)^{-1}$
is homotopic to $m\gamma_1$ for a certain $m$. After this, the above equality follows easily from the construction of the chains $Z_i$.
\end{proof}

\subsection{A system of generators of $H_{n-2}(\calX_i;\ZZ)$}
\label{gusanosz}
To lift the computation to coefficients in $\ZZ$, we need to compute the integer homology of $\calA_i\cap\calB$. 
We can do so by computing the long exact sequence of the pair $(\calA_i\cap\calB,\calA_i\cap\calB_u)$ using the same arguments that 
we used to compute the long exact sequence of the pair $(\calB,\calB_u)$. We obtain:

\[H_{n-1}(\calA_i\cap\calB;\ZZ)\cong\ZZ_2,\quad H_{n-3}(\calA_i\cap\calB;\ZZ)\cong\ZZ_2,\]
\[H_3(\calA_i\cap\calB;\ZZ)\cong\ZZ,\quad H_1(\calA_i\cap\calB;\ZZ)\cong\ZZ_2,\quad H_0(\calA_i\cap\calB;\ZZ)\cong\ZZ.\]
and zero otherwise.

With these data, and the universal coefficients theorem, we can compute the Mayer-Vietoris sequence (\ref{mayerviet-ABXz2}) with coefficients in $\ZZ$:

 \[\label{mayerviet-ABXz}
\xymatrix@R=1ex@C=3ex{{\begin{array}{c} \\  \\ \oplus_i H_n(\calA_i\cap\calB;\ZZ) \\ \rotatebox{90}{$\cong$} \\ 0 \end{array}}  \ar@{^{(}->}[r] & {  \begin{array}{c@{}c@{}c} {} & & \\ & & \\  \oplus _i H_n(\calA_i;\ZZ) & \bigoplus  & H_n(\calB;\ZZ)\\ \rotatebox{90}{$\cong$} & & \rotatebox{90}{$\cong$} \\ 0 & & 0\end{array}  } \ar[r] & {\begin{array}{c} \\ \\ H_n(\fraX;\ZZ)\\ \rotatebox{90}{$\cong$} \\ 0 \end{array}} \ar `d[l] `l[dll] [dll]+<0ex,3ex>*{}\\
{\begin{array}{c} \\ \\ \oplus_i H_{n-1}(\calA_i\cap\calB;\ZZ) \\ \rotatebox{90}{$\cong$}\\ \ZZ_2^a \end{array}}  \ar[r] & {  \begin{array}{c@{}c@{}c} & & \\{} & & \\ \oplus_i H_{n-1}(\calA_i;\ZZ) & \bigoplus & H_{n-1}(\calB;\ZZ) \\ \rotatebox{90}{$\cong$} & & \rotatebox{90}{$\cong$}\\ 0 & \ZZ_2^a&\oplus \ZZ^{\mu_1-2a+1} \end{array}  } \ar[r] & {\begin{array}{c}\\ \\ H_{n-1}(\fraX;\ZZ)\\ \rotatebox{90}{$\cong$} \\ \ZZ^{\mu_1-2a+1} \end{array}} \ar `d[l] `l[dll] [dll]+<0ex,3ex>*{}\\
{\begin{array}{c} \\ \\ \oplus_i H_{n-2}(\calA_i\cap\calB;\ZZ)\\ \rotatebox{90}{$\cong$} \\ 0 \end{array}}  \ar[r] & {  \begin{array}{c@{}c@{}c} & &  {}\\ & & \\ \oplus_i H_{n-2}(\calA_i;\ZZ) & \bigoplus & H_{n-2}(\calB;\ZZ) \\ \rotatebox{90}{$\cong$} & & \rotatebox{90}{$\cong$}\\ \ZZ^{a} &  & 0 \end{array}  } \ar[r] & {\begin{array}{c}\\ \\ H_{n-2}(\fraX;\ZZ)\\ \rotatebox{90}{$\cong$} \\ \ZZ^{a} \end{array}} \ar`d[l] `l[dll] [dll]+<0ex,3ex>*{} \\
{\begin{array}{c} \\ \\ \oplus_i H_{n-3}(\calA_i\cap\calB;\ZZ)\\ \rotatebox{90}{$\cong$} \\ \ZZ_2^a \end{array}}  \ar[r] & {  \begin{array}{c@{}c@{}c} & &  {}\\ & & \\ \oplus_i H_{n-3}(\calA_i;\ZZ) & \bigoplus & H_{n-3}(\calB;\ZZ) \\ \rotatebox{90}{$\cong$} & & \rotatebox{90}{$\cong$}\\ 0 &  & \ZZ_2 \end{array}  } \ar@{>>}[r] & {\begin{array}{c}\\ \\ H_{n-3}(\fraX;\ZZ)\\ \rotatebox{90}{$\cong$} \\ 0 \end{array}}}
\]

The non-zero lower homology groups are isomorphic to those of $\Sigma[1]_s$.

We give a sufficient system of generators of $H_{n-2}(\calX;\ZZ)$. For any $i$ choose an $(n-2)$-sphere $S_i$ in $\calA_i$ 
generating $H_{n-2}(\calA_i;\ZZ)$. Choosing the
orientations of the summands of $Z_i$ appropiately it turns out that we have a $\ZZ$-closed chain. It is clear that 
$[Z_2],...,[Z_a]$ generate the kernel
of the first homomorphism of the $(n-3)$-row of the Mayer-Vietoris sequence. The image of the second morphism of the $(n-2)$-row 
is obviously generated by the $(n-2$-spheres $S_i$..

\section{Homology of $\calM$}
\label{seccalm}

\subsection{Coefficients in $\ZZ_2$}
Recall that $\Sigma_s$ is the Milnor fibre of $\Sigma$, and has the homotopy type of a bouquet of $\mu_0$ spheres. The functions 
$g_1,...,g_{n-3},det(H(f))$ define a $2$-dimensional i.c.i.s. $\Sigma[2]_0$ with Milnor number $\mu_1$ (see Definition~\ref{defnueva}). 

The restriction 
\[det(H(f_s))|_{\Sigma_s}:\Sigma_s\to\CC\]
has isolated critical points. Therefore, taking $\eta$ so small that the disk $D_\eta$ only contains $0$ as critical value of 
the restriction, the set $\Sigma_s$ 
is homotopy equivalent to the result of attaching to $(det(H(f_s))|_{\Sigma_s})^{-1}(D_\eta)$ the Lefschetz thimbles associated to 
the critical points of $det(H(f_s))|_{\Sigma_s}$ not
contained in the zero level. There are exactly $\mu_0+\mu_1-a$ such Lefschetz thimbles (see~\cite{Lo}). 
Since the Lefschetz thimbles are $3$-disks they are attached 
along $2$-spheres to the boundary of $(det(H(f_s))|_{\Sigma_s})^{-1}(D_\eta)$, which is $5$-dimensional. 
Hence, a transversality argument ensures that all the attaching spheres are disjoint. Denote by $C_1,...,C_{\mu_0+\mu_1-a}$ the 
Lefschetz thimbles. We have found a homotopy equivalence 
\begin{equation}
\label{eqretraction1}
M':=(det(H(f_s))|_{\Sigma_s})^{-1}(D_\eta)\bigcup (\cup_{i=1}^{\mu_0+\mu_1-a}C_i)\hookrightarrow \Sigma_s,
\end{equation}
which in fact (since we are working with $CW$-complexes) is a deformation retract.

Since we have a locally trivial fibration
\begin{equation}
\label{fibracionsn-4}
\pi:\calM\setminus\pi^{-1}(det(H(f_s))^{-1}(0))\to\Sigma_s\setminus det(H(f_s))^{-1}(0)
\end{equation}
we can lift the deformation retract~(\ref{eqretraction1}) to a deformation retract
\begin{equation}
\label{eqretraction2}
\calM':=\pi^{-1}(M')\hookrightarrow \calM.
\end{equation}

We will compute the homology of $\calM'$ using a Mayer-Vietoris sequence. By the previous deformation retract we identify the 
homology of $\calM'$ and $\calM$. Denote $\pi^{-1}(C_i)$ by $\calC_i$. Since $C_i$ is contractible the fibration over it 
is trivial, and, hence, $\calC_i$ and $\pi^{-1}(\partial C_i)$ are homotopy equivalent to $C_i\times\SSS^{n-4}$, and 
$\partial C_i\times\SSS^{n-4}\cong\SSS^2\times\SSS^{n-4}$. Decompose $\calM'$ as
\begin{equation}
\label{descom}
\calM'=\calX\bigcup(\cup_{i=1}^{\mu_0+\mu_1-a}\calC_i).
\end{equation}
The associated Mayer-Vietoris sequence (with coefficients in $\ZZ_2$) is:

\[\label{mayervietMzeta2}
\xymatrix@R=1ex@C=3ex{{\begin{array}{c} \\  \\ \oplus_i H_{n-1}(\pi^{-1}(\partial C_i);\ZZ_2) \\ \rotatebox{90}{$\cong$} \\ 0 \end{array}}  \ar@{^{(}->}[r] & {  \begin{array}{c@{}c@{}c} {} & & \\ & & \\  \oplus _i H_{n-1}(\calC_i;\ZZ_2) & \bigoplus  & H_{n-1}(\calX;\ZZ_2)\\ \rotatebox{90}{$\cong$} & & \rotatebox{90}{$\cong$} \\ 0 & & \ZZ_2^{\mu_1-2a+1}\end{array}  } \ar[r] & {\begin{array}{c} \\ \\ H_{n-1}(\calM;\ZZ_2)\\ \rotatebox{90}{$\cong$} \\ \ZZ_2^{\mu_0+2\mu_1-4a+1+e} \end{array}} \ar `d[l] `l[dll] [dll]+<0ex,3ex>*{} \\
{\begin{array}{c} \\ \\ \oplus_i H_{n-2}(\pi^{-1}(\partial C_i);\ZZ_2) \\ \rotatebox{90}{$\cong$}\\ \ZZ_2^{\mu_0+\mu_1-a} \end{array}}  \ar[r]^-{\varphi_{n-2}} & {  \begin{array}{c@{}c@{}c} & & \\{} & & \\ \oplus_i H_{n-2}(\calC_i;\ZZ_2) & \bigoplus & H_{n-2}(\calX;\ZZ_2) \\ \rotatebox{90}{$\cong$} & & \rotatebox{90}{$\cong$}\\ 0 & &\ZZ_2\oplus \ZZ_2^{a-1} \end{array}  } \ar[r] & {\begin{array}{c}\\ \\ H_{n-2}(\calM;\ZZ_2)\\ \rotatebox{90}{$\cong$} \\ \ZZ_2^{e} \end{array}} \ar `d[l] `l[dll] [dll]+<0ex,3ex>*{} \\
{\begin{array}{c} \\ \\ \oplus_i H_{n-3}(\pi^{-1}(\partial C_i);\ZZ_2)\\ \rotatebox{90}{$\cong$} \\ 0 \end{array}}  \ar[r] & {  \begin{array}{c@{}c@{}c} & &  {}\\ & & \\ \oplus_i H_{n-3}(\calC_i;\ZZ_2) & \bigoplus & H_{n-3}(\calX;\ZZ_2) \\ \rotatebox{90}{$\cong$} & & \rotatebox{90}{$\cong$}\\ 0 &  & 0 \end{array}  } \ar[r] & {\begin{array}{c}\\ \\ H_{n-3}(\calM;\ZZ_2)\\ \rotatebox{90}{$\cong$} \\ 0 \end{array}} \ar `d[l] `l[dll] [dll]+<0ex,3ex>*{} \\
{\begin{array}{c} \\ \\ \oplus_i H_{n-4}(\pi^{-1}(\partial C_i);\ZZ_2)\\ \rotatebox{90}{$\cong$} \\ \ZZ_2^{\mu_0+\mu_1-a} \end{array}}  \ar[r] & {  \begin{array}{c@{}c@{}c} & &  {}\\ & & \\ \oplus_i H_{n-4}(\calC_i;\ZZ_2) & \bigoplus & H_{n-4}(\calX;\ZZ_2) \\ \rotatebox{90}{$\cong$} & & \rotatebox{90}{$\cong$}\\ \ZZ_2^{\mu_0+\mu_1-a} &  & 0 \end{array}  } \ar@{>>}[r] & {\begin{array}{c}\\ \\ H_{n-4}(\calM;\ZZ_2)\\ \rotatebox{90}{$\cong$} \\ 0 \end{array}}}
\]
for some $e\in\NN$.

We will find out what are the possible values for $e$. We have given a basis $\{Z_i\}_{i=1}^a$ of $H_{n-2}(\calX;\ZZ_2)$ 
in~\ref{gusanosz2}.

\begin{lema}
\label{imagen}
The composition
\[\tau:\oplus_i H_{n-2}(\pi^{-1}(\partial C_i);\ZZ_2)\stackrel{\varphi_{n-2}}{\to}H_{n-2}(\calX;\ZZ_2)\to H_{n-2}(\calX;\ZZ_2)/([Z_1])\]
is surjective.
\end{lema}
\begin{proof}
For each of the $a$ singular points $p_i$ of $det(H(f_s))^{-1}(0)$ there is a vanishing cycle $E_i$ which 
is a embedded $2$-sphere in $det(H(f_s))^{-1}(u)$. The parameters $s,u,\zeta$ (see Section~\ref{sectiondecomp}) 
can be choosen so that $det(H(f_s))^{-1}(u)\cap A_i(\zeta)$ is a tubular neighbourhood of $E_i$ in $det(H(f_s))^{-1}(u)$.

The sphere $\partial C_k$ can be choosen to be embedded in $det(H(f_s))^{-1}(u)$ and, after a perturbation, transverse to $E_i$ 
for any $i$. Let 
\[\iota_k:\partial C_k\hookrightarrow det(H(f_s))^{-1}(u)\]
denote the embedding. Let $b_{k,i}$ the number of intersection points of $\partial C_k$ and $E_i$.
Choosing the tubular neighbourhoods 
of $E_i$ small enough we find that $\iota_k^{-1}(A_i(\zeta))$ is a disjoint union of disks $D_{k,i,j}$ with $j\in\{1,...,b_{k,i}\}$,
and the boundary of each of them represents the generator of $H_1(B_u;\ZZ)$. By Remark~\ref{isorestriccion} the number 
$b_k:=\sum_{i}b_{k,i}$ is even: otherwise the image in $H_1(B_u;\ZZ)$ of the boundary
\[\partial (C_k\setminus (\cup_{i,j}D_{k,i,j}))\] would
be a non-zero homology class. 
We claim the following equality
\begin{equation}
\label{eqn1}
\tau([\pi^{-1}(\partial C_k)])=\sum_{i=1}^ab_{k,i}[Z_i].
\end{equation}
Let us finish the proof assuming this claim.

Any Lefschetz thimble $C_k$ gives rise to a class $[\partial C_k]\in H_2(det(H(f_s))^{-1}(D_\xi);\ZZ)$. It is easy to check that
its image by the connecting homomorphism $\delta_2$ is equal to
\begin{equation}
\label{eqn2}
\delta_2([\partial C_k])=\sum_{i=1}^ab_{k,i}[\psi_i]=\sum_{i=2}^ab_{k,i}([\psi_i]-[\psi_1])
\end{equation}
where $\psi_i$ is a generator of $H_1(\partial A_i(\zeta)\cap B;\ZZ)$ for any $i$. The first equality is by connstruction of 
the connecting homomorphism and the second is true because $\sum_{i}b_{k,i}$ is even and, hence we
have the equality $b_{k,1}=\sum_{i=2}^ab_{k,i}$ in $\ZZ_2$.

Let $\alpha_1$ be the first mapping of the $1$-row of the sequence~(\ref{MayerVietabajo}). Define the isomorphism
\[\theta:H_{n-2}(\calX,\ZZ_2)/([Z_1])\to ker(\alpha_1)\]
given by $\theta([Z_i]):=[\psi_i]-[\psi_1]$. Any element $[Z']\in H_{n-2}(\calX;\ZZ_2)/([Z_1])$ corresponds to  
an element in $ker(\alpha_1)$, which is the image by $\delta_2$ of a class $[Y]\in H_2(det(H(f_s))^{-1}(D_\xi);\ZZ)$. 
Such a class can be expressed as a sum $$[Y]=\sum_{k=1}^{\mu_0+\mu_1-a}m_k[\partial C_k].$$ The concidence of the coefficients
in the last terms of equations~(\ref{eqn1})~and~(\ref{eqn2}) give the equality $\tau([Y])=[Z']$.

Now we prove the claim.
Choose a point $x_0\in \partial C_k\setminus\cup_{i,j}D_{k,i,j}$ and choose a disk $D_0$ around it in $\partial C_k$ 
disjoint to the disks $D_{k,i,j}$. Deform the immersion $\iota_k|_{D_0}$ so that the embedding of its boundary remains fixed,
it meets $E_1$ transversely precisely at $b_k$ points, all different from $x_0$, and it is disjoint from $E_j$ for any 
$j\neq 1$. After this deformation the intersection $\iota_k|_{D_0}^{-1}(A_1(\zeta))$ consists of $b_k$ disjoint disks 
$\{D'_{k,i,j}\}_{i\in\{1,..,a\},j\in\{1,...,b_{k,i}\}}$ (we choose
the indexing to make it easy to make a bijection with the disks $D_{k,i,j}$).

Choose non-intersecting paths $\alpha_{k,i,j}$ in $D_0\setminus (\cup_{i,j}\dot{D'}_{k,i,j})$
joining $x_0$ with a point $y_{k,i,j}\in \partial D'_{k,i,j}$. Choose non-intersecting paths 
\[\beta_{k,i,j}:[0,1]\to\partial C_k\setminus(\bigcup_{k,i,j}(\alpha_{k,i,j}([0,1])\cup\dot{D'}_{k,i,j}\cup\dot{D}_{k,i,j})\]
joining $\partial D'_{k,i,j}$ with $\partial D_{k,i,j}$. For a schematic picture, see Figure~\ref{figura3}.

\begin{figure}[h]
\setlength{\unitlength}{0.00087489in}
\begin{picture}(4552,5045)

\put(3311,3701){$D_0$}
\put(2186,4241){$x_0$}
\put(2096,2891){$\beta_{k,i,j}$}
\put(1331,4971){$\alpha_{k,i,j}$}
\put(2141,3476){$D'_{k,i,j}$}
\put(1916,2261){$D_{k,i,j}$}

 \includegraphics{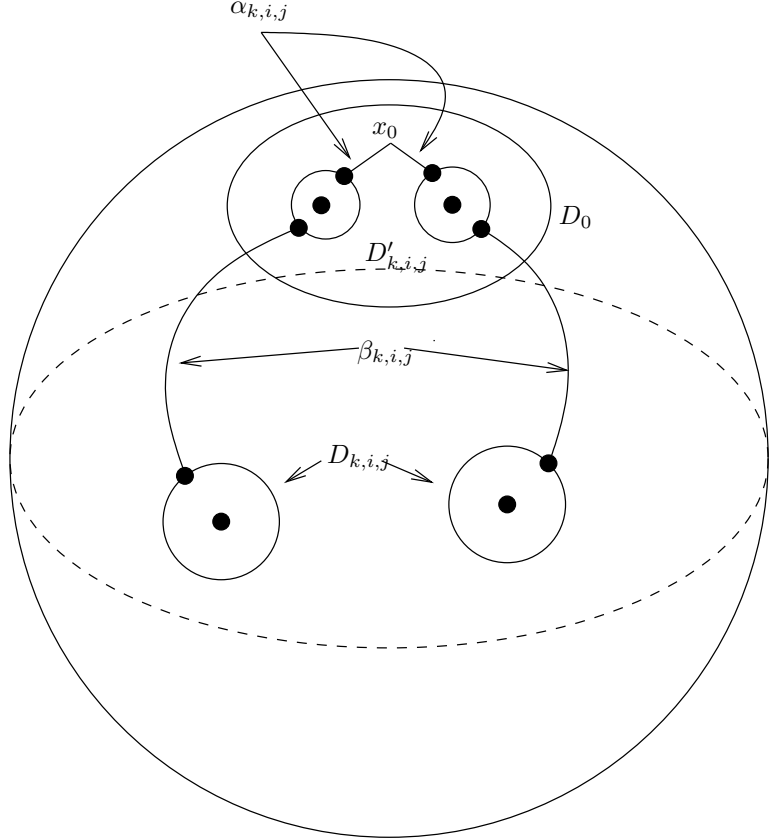}
\end{picture}

\caption{\label{figura3}The system of paths in $\partial C_k$}

\end{figure}

The complement of $\bigcup_{k,i,j}(\beta_{k,i,j}([0,1])\cup\alpha_{k,i,j}([0,1])\cup {D'}_{k,i,j}\cup {D}_{k,i,j})$
is a topological disk $G$. Since $G$ is contained in $B_u$ we can restrict the fibration~(\ref{fibpar}) to $G$ and obtain
a trivial fibration of pairs with fibre homotopic to
$(\SSS^{n-3},\SSS^{n-4})$, with $\SSS^{n-4}$ embedded as an equator. Consider a mapping 
\[\sigma:G\times \SSS^{n-3}\to\calB\]
such that $\sigma(\{g\}\times\SSS^{n-3})$ generates the $(n-3)$-homology of the fibre over $\iota(g)$ 
by the fibration~(\ref{fibpar}). Denote by $H^+$ one hemisphere of $\SSS^{n-3}$.
The restriction
\[\psi:G\times H^+\to\calB\subset\calX\]
defines a singular chain in $\calX$.

Let $Z'_{k,i,j}$ be the chain associated to $\beta_{k,i,j}$ by the procedure given in~\ref{gusanosz2}.
Adding and substracting $\SSS^{n-2}$-hemispheres $K_{k,i,j}$ and $K'_{k,i,j}$ for any $i,j$ (see the procedure in~\ref{gusanosz2}),
the chain $\partial\calC_k+\partial\psi$ is shown to be equal to a sum
\[\sum_{i,j}Z'_{i,j}+\sum_{i=1}^aY_i\]
where $Y_i$ is a closed chain contained in $\calA_i$.

The $(n-2)$-row of the sequence~(\ref{mayerviet-ABXz2}) shows that, for any $i$, any class in $H_{n-2}(\calX;\ZZ_2)$ supported by 
a chain contained in $\calA_i$ is a multiple of $[Z_1]$. On the other hand, by Lemma~\ref{cambiocamino} there exists 
$c_{i,j}\in\ZZ_2$ such that $[Z'_{i.j}]=[Z_{i}]+c_{i,j}[Z_1]$. This proves the claim.
\end{proof}

This means that the only possible values for $e$ are $0$ and $1$. We will now characterize the cases in which each value is obtained.

\begin{lema}\label{cork3e0}
 If $corank(H[f_0](0))\geq 3$, then $e=0$.
\end{lema}
\begin{proof}
Consider the unfolding
\begin{equation}
\label{unfolding11}
F(x_1,...,x_n,b,(c_{i,j})):=(G_{1,b},...,G_{n-3,b})(h_{i,j}+c_{i,j})(G_{1,b},...,G_{n-3,b})^t.
\end{equation}
given in~(\ref{unfolding1}). If $corank(H[f_0](0))\geq 3$ there exists a parameter $s_0\in S$ and a point $x\in\Sigma_{s_0}$ 
such that the germ $f_{s_0}$ at $x$ is right-equivalent to a germ of the form:
\[
 (y_1,y_2,y_3,\ldots)\cdot\left(
\begin{array}{c|c} \begin{array}{ccc} l_1 & l_2 & l_3 \\ 
l_2 & l_4 & l_5 \\ l_3 & l_5 & l_6 \end{array} & 0 \\
\hline
0 & Id \end{array}\right) \cdot \left(\begin{array}{c} y_1 \\ y_2 \\ y_3 \\ \vdots \end{array} \right),
\] where the $l_i$'s are generic linear forms and the $y_i$'s are variables. 
The Milnor fibre of such germ function is the suspension of the Milnor fibre $M$ of

\[
 (y_1,y_2,y_3)\cdot\left(
\begin{array}{ccc} l_1 & l_2 & l_3 \\ 
l_2 & l_4 & l_5 \\ l_3 & l_5 & l_6 \end{array}\right) \cdot \left(\begin{array}{c} y_1 \\ y_2 \\ y_3 \end{array} \right),
\] 
and this one can be computed by projecting to the variables $(y_1,y_2,y_3)$. This projection is a fibration over 
$\CC^3\setminus \{0\}\approx\SSS^5$ whose fibre is contractible for being given by the solutions of a system of linear equations.
So $M$ has the homotopy type of $\SSS^5$ and $H_4(M)=0$. The general case is a suspension of this one.
Hence the Milnor fibre of the germ $f_{s_0}$ at $x$ is homotopic to $\SSS^{n-1}$ and its $(n-2)$-homology vanishes.

Since $\Sigma_{s_0}$ is smooth at $x$ its versal deformation is trivial. Hence the unfolding given by~(\ref{unfolding1}) for the 
germ $(f_{s_0})_x$ is of the form:
\begin{equation}
\label{unfoldingaux}
F(x_1,...,x_n,(c_{i,j})):=(G_{1,s_0},...,G_{n-3,s_0})(h_{i,j}+c_{i,j})(G_{1,s_0},...,G_{n-3,s_0})^t.
\end{equation}
with $(c_{i,j})\in SM(n-3)$. Observe that this unfolding can be obtained by pullback and localising near $x$ from the unfolding
~(\ref{unfolding11}).
Let $B(x,\epsilon_0)$ be a Milnor ball for $(f_{s_0})_x$ contained in the Milnor ball $B_\epsilon$ of $f$. If $s$ is generic and
very close to $s_0$ and $t$ is small enough then
\begin{equation}
\label{inclusionmilnorfibres}
\iota:f_{s}^{-1}(t)\cap B(x,\epsilon_0)\hookrightarrow f_{s}^{-1}(t)\cap B_\epsilon
\end{equation}
is an inclussion of the Milnor fibre of $(f_{s_0})_x$ into the Milnor fibre of $f$.

Since $corank(H(f_{s_0}))(x)\geq 2$, if $s$ is close to $s_0$ there exists at least a point $p_i$ of $\Sigma[2]_s$ contained in 
$B(x,\epsilon_0)$. If $e\neq 0$, that is, if $H_{n-2}(\calM;\ZZ_2)\neq 0$ then, by Proposition~\ref{homsplit} with coefficients in 
$\ZZ_2$ there is a $\SSS^{n-2}$ in $\calA_i\subset f_s^{-1}(t)\cap B_\epsilon$ representing a non-trivial homology class 
in the Milnor fibre of $f$. But this is impossible because $\calA_i$ is already contained in $f_s^{-1}(t)\cap B(x,\epsilon_0)$
and in this space there is no non-trivial $(n-2)$-homology.
\end{proof}

Now we will see that, in the case where $corank(H[f_0](0))=2$, the number $e$ turns out to be $1$. 
Recall that $\calA_i$ is homotopic to $\SSS^{n-2}$, denote by $Z$ the generator of $H_{n-2}(\calA_i;\ZZ)$. 
Since we have an inclusion $i:\calA_i\hookrightarrow\calM$, we need to check that $i_*(Z)\neq 0$.

\begin{lema}\label{casosimplee1}
 If $f$ is of the form
\[
 f=(g_1,g_2)\cdot \left(\begin{array}{cc} h_{1,1} & h_{1,2} \\ h_{1,2} & h_{2,2} \end{array}\right) \cdot \left(\begin{array}{c}g_1 \\ g_2 \end{array}\right)
\]
then $i_*(Z)\neq 0$

\end{lema}
\begin{proof}
 Let $\omega$ be a closed differential form defined in $\CC^2\setminus\{0\}$ such that $\int_{\SSS^3}\omega\neq 0$ 
for a sphere $\SSS^3$ around the origin in $\CC^2$. Consider the map
\[\phi:\CC^n\to\CC^2\]
defined by $\phi:=(g_1,g_2)$.

Then $\Omega:=\phi^\star \omega$ is a closed differential form defined over the Milnor fibre of $f$. The change of variables formula gives 
the inequality $\int_{i_\star (Z)}\phi^\star \omega\neq 0$
\end{proof}

Now let's generalize this argument for the case where the corank is two, but the dimension is higher:

\begin{lema}
\label{cork2e1}
If $corank(H[f_0](0))=2$, then $e=1$.
\end{lema}

\begin{proof}
We may assume (see~\ref{corrango2}) that $f$ is of the form 
\[
 f=(g_1,g_2)\cdot \left(\begin{array}{cc} h_{1,1} & h_{1,2} \\ h_{1,2} & h_{2,2} \end{array}\right) \cdot \left(\begin{array}{c}g_1 \\ g_2 \end{array}\right) +g_3^2+\cdots +g_{n-3}^2.
\]

We consider the unfolding $F$ defined in~(\ref{unfoldinggrande}). Clearly there are parameters $s$ such that the functions
$(h_{1,1,s},h_{1,2,s},h_{2,2,s},G_{1,s},\ldots ,G_{n-3,s})$ vanish at the origin and form a holomorphic coordinate system around 
it. In this case the local Milnor fibre of the deformed function $F_s$ at the origin has the homotopy type of a $(n-2)$-sphere. 

Let $\epsilon$ and $\delta$ be radii for the Milnor fibration of $f$. Let $\epsilon'$ and $\delta'$ be radii for the Milnor fibration of
$f_s$ at the origin. By Theorem~\ref{teounfolding1} we have that $F_s^{-1}(\delta')\cap B_\epsilon$ is diffeomorphic to the Milnor fibre of $f$.
Let $Z$ be a cycle in the local Milnor fibre $F_s^{-1}(\delta')\cap B_{\epsilon'}$ generating the $(n-2)$-homology group.
In order to show that $e=1$ it is enough to show that the homology class $[Z]$ is nonzero considered in the bigger space
$F_s^{-1}(\delta')\cap B_\epsilon$. For this it suffices to find a closed $(n-2)$-differential form $\Omega$, defined in 
$F_s^{-1}(\delta')\cap B_\epsilon$ such that $\int_{Z}\Omega\neq 0$.

In order to define such a form, choose a positive 
function $\beta:\CC\mapsto \RR\subseteq \CC$ such that $\beta|_{D(0,\eta/2)}\equiv 0$ and 
$\beta|_{\CC\setminus D(0,\eta)}\equiv 1$ for a sufficiently small radius $\eta$. 
Now take $\omega$ a closed $3$-form in $\CC^2\setminus \{0\}$
that generates the de Rham cohomology in degree $3$.

We have the function
\[
\begin{array}{ccc}
\psi:\CC^n\setminus V(G_{1,s},G_{2,s}) & \longrightarrow  & \CC^2\setminus\{0\} \\
x & \longmapsto & (G_{1,s}(x),G_{2,s}(x))
\end{array}.
\]
Define
\[
 \Omega:=\psi^*\omega\wedge \beta(G_{3,s}^2+\cdots +G_{n-3,t}^2-\delta')dG_{3,s}\wedge\cdots\wedge dG_{n-3,s}.
\]

Let us check that $\Omega$ is defined in all $F_s^{-1}(\delta')\cap B_{\epsilon}$: the form $\psi^\star\omega$ is only defined in 
$\CC^n\setminus V(G_{1,s},G_{2,s})$, but the factor $\beta(G_{3,s}^2-\cdots -G_{n-3,s}^2-\delta')$ is identically zero when 
$G_{3,s}^2-\cdots -G_{n-3,s}^2-\delta'$ is small enough.

In order to check that $\Omega$ is closed notice that since $\omega$ is closed, so is $\psi^*\omega$. Hence it is sufficient to
show that
\[\beta(G_{3,s}^2+\cdots +G_{n-3,t}^2-\delta')dG_{3,s}\wedge\cdots\wedge dG_{n-3,s}\]
is closed. A chain rule argument shows the equality:
\[d\beta(G_{3,s}^2+\cdots +G_{n-3,t}^2-\delta')=\frac{\partial \beta}{\partial z}(G_{3,s}^2-\cdots -G_{n-3,s}^2)\sum_{i=3}^{n-3}2G_{i,s}dG_{i,s}\]
which means that
\[d\Omega=\psi^\star w\wedge \frac{\partial \beta}{\partial z}(G_{3,s}^2+\cdots +G_{n-3,t}^2-\delta')(\sum_{i=3}^{n-3}2G_{i,s}dG_{i,s})\wedge dG_{3,s}\wedge\cdots\wedge dG_{n-3,s}=0.
\]

Finally we will check that the form $\Omega$ integrated against the cycle $Z$ gives a non-zero result. We start by
giving an explicit description of $Z$. Let $Z'$ be the cycle that generates the $3$-homology of 
$$
\{(G_{1,s},G_{2,s})(h_{i,j,s})(G_{1,s},G_{2,s})^t=\delta'\}\cap B_{\epsilon'}.
$$
Define the function 
\[
\begin{array}{ccc}
 \alpha_w:\CC^5 & \longrightarrow  & \CC^5 \\
(x_1,x_2,x_3,x_4,x_5) & \longmapsto & \sqrt[3]{\frac{w}{3}}(x_1,x_2,x_3,x_4,x_5)
\end{array}.
\]
The cycle $Z$ admits the following parametrisation: since the functions
\[(h_{1,1,s},h_{1,2,s},h_{2,2,s},G_{1,s},\ldots ,G_{n-3,s})\]
form a holomorphic coordinate system at the origin the 
vanishing of the first $5$ of them defines a germ $(M,O)$ of $(n-5)$-dimensional complex manifold at the origin. 
Let $B(0,\sqrt{\delta'})$ denote the ball of radius $\sqrt{\delta'}$ centered at the origin of $\RR^{n-5}$, being 
$\RR^{n-5}$ the real locus of $(M,O)$. The parametrisation is given by
\[
\begin{array}{ccc}
Z'\times B(0,\sqrt{\delta'}) & \longrightarrow  & Z\\
(p,y_3,\ldots y_{n-3}) & \longmapsto & (\alpha_{\delta'-y_3^2-\cdots -y_{n-3}^2}(p),y_3,\ldots y_{n-3})
\end{array}.
\]

Now we compute $\int_{Z}\Omega$ using Fubini's theorem:
\[
 \int_{Z}\Omega=\int_{Z'\times B(0,\sqrt{\delta'})}\psi^* \omega\wedge \beta(G_{3,s}^2+\cdots+G_{n-3,s}^2-\delta')dG_{3,s}\wedge\cdots\wedge G_{n-3,s}=
\]
\[
=\int_{B(0,\sqrt{\delta'})} \beta(G_{3,s}^2-\cdots+G_{n-3,s}^2-\delta')\cdot\int_{Z'}\alpha_{G_{3,s}^2+\cdots+G_{n-3,s}^2-\delta'}^*\psi^*\omega dG_{3,s}\wedge\cdots\wedge dG_{n-3,s}=
\]
\[=\int_{B(0,\sqrt{\delta'})}\beta(G_{3,s}^2+\cdots+G_{n-3,\delta'}^2-\delta')dG_{3,s}\wedge\cdots\wedge dG_{n-3,s}
\]
which is nonzero.
\end{proof}

Now we can prove easily Lemma~\ref{postpuesto} using an example:
\begin{proof}[Proof of Lemma~\ref{postpuesto}]
Since we are with coefficients in $\ZZ_2$, if the mapping is not an isomorphism then it is identically zero.

The function $f:\CC^5\to\CC$ given by
\[
 f=(x_1,x_2)\cdot \left(\begin{array}{cc} x_3 & x_4 \\ x_4 & x_3-x_5^2 \end{array}\right) \cdot \left(\begin{array}{c}g_1 \\ g_2 \end{array}\right)
\]
has finite extended codimension with respect to $(x_1,x_2)$. By a procedure simmilar to the one we have used to compute
the homotopy type of the Milnor fibre of the $D(3,2)$ singularity, in~\cite{FdB3} we have shown that the Milnor fibre
of $f$ at the origin is homotopy equivalent to $\SSS^3$.

If we take a generic parameter $s$ of the unfolding $F$ 
associated to $f$ in Section~\ref{sectionunfoldings} we see that $F_s$ has no Morse points outside $\Sigma_s=\Sigma_0$,
there are precisely $2$ points of type $D(3,2)$, and the Milnor number of the i.c.i.s. $\Sigma_0\cap\{det(H(f))=0\}$ is 
equal to $3$. Let us assume that the mapping in the statement of Lemma~\ref{postpuesto} is identically zero.
In this case the previous long exact sequences can be used to compute the homology of the Milnor fibre of $f$, and they
give that the $4$-homology group is non-zero. This gives a contradiction.
\end{proof}

\subsection{Integral coefficients}
From the integer homology of $\calX$, it is easy to see by the Mayer-Vietoris sequence that $H_k(\calM;\ZZ)=H_k(\calX;\ZZ)$ for $k\neq n-1,n-2$. 

On the other hand, the group $H_{n-1}(\calM;\ZZ)$ is torsion free since $\calM$ is a $(n-1)$-dimensional Stein
space. By the Universal Coefficients Theorem and our computation of homology with coeffficients in $\ZZ_2$, it is 
easily obtained that $H_{n-2}(\calM;\ZZ)$ has no 2-torsion: as we have seen before, when $e=1$, the $\ZZ_2$ component of $H_{n-2}(\calM;\ZZ_2)$ is represented by a torsion free class (its integral against a closed form is non-zero), and hence it comes from a $\ZZ$ component in $H_{n-2}(\calM;\ZZ)$.

Summarising, the Mayer-Vietoris sequence with coefficients in $\ZZ$ is as follows:
\[\label{mayervietMzeta}
\xymatrix@R=1ex@C=3ex{{\begin{array}{c} \\  \\ \oplus_i H_{n-1}(\pi^{-1}(\partial C_i);\ZZ) \\ \rotatebox{90}{$\cong$} \\ 0 \end{array}}  \ar@{^{(}->}[r] & {  \begin{array}{c@{}c@{}c} {} & & \\ & & \\  \oplus _i H_{n-1}(\calC_i;\ZZ) & \bigoplus  & H_{n-1}(\calX;\ZZ)\\ \rotatebox{90}{$\cong$} & & \rotatebox{90}{$\cong$} \\ 0 & & \ZZ^{\mu_1-2a+1}\end{array}  } \ar[r] & {\begin{array}{c} \\ \\ H_{n-1}(\calM;\ZZ)\\ \rotatebox{90}{$\cong$} \\ \ZZ^{\mu_0+2\mu_1-4a+1+e} \end{array}} \ar `d[l] `l[dll] [dll]+<0ex,3ex>*{} \\
{\begin{array}{c} \\ \\ \oplus_i H_{n-2}(\pi^{-1}(\partial C_i);\ZZ) \\ \rotatebox{90}{$\cong$}\\ \ZZ^{\mu_0+\mu_1-a} \end{array}}  \ar[r]^-{\varphi_{n-2}} & {  \begin{array}{c@{}c@{}c} & & \\{} & & \\ \oplus_i H_{n-2}(\calC_i;\ZZ) & \bigoplus & H_{n-2}(\calX;\ZZ) \\ \rotatebox{90}{$\cong$} & & \rotatebox{90}{$\cong$}\\ 0 & &\ZZ^a \end{array}  } \ar[r] & {\begin{array}{c}\\ \\ H_{n-2}(\calM;\ZZ)\\ \rotatebox{90}{$\cong$} \\ \ZZ_2^{e}\oplus T \end{array}} \ar `d[l] `l[dll] [dll]+<0ex,3ex>*{} \\
{\begin{array}{c} \\ \\ \oplus_i H_{n-3}(\pi^{-1}(\partial C_i);\ZZ)\\ \rotatebox{90}{$\cong$} \\ 0 \end{array}}  \ar[r] & {  \begin{array}{c@{}c@{}c} & &  {}\\ & & \\ \oplus_i H_{n-3}(\calC_i;\ZZ) & \bigoplus & H_{n-3}(\calX;\ZZ) \\ \rotatebox{90}{$\cong$} & & \rotatebox{90}{$\cong$}\\ 0 &  & 0 \end{array}  } \ar[r] & {\begin{array}{c}\\ \\ H_{n-3}(\calM;\ZZ)\\ \rotatebox{90}{$\cong$} \\ 0 \end{array}} \ar `d[l] `l[dll] [dll]+<0ex,3ex>*{} \\
{\begin{array}{c} \\ \\ \oplus_i H_{n-4}(\pi^{-1}(\partial C_i);\ZZ)\\ \rotatebox{90}{$\cong$} \\ \ZZ^{\mu_0+\mu_1-a} \end{array}}  \ar[r] & {  \begin{array}{c@{}c@{}c} & &  {}\\ & & \\ \oplus_i H_{n-4}(\calC_i;\ZZ) & \bigoplus & H_{n-4}(\calX;\ZZ) \\ \rotatebox{90}{$\cong$} & & \rotatebox{90}{$\cong$}\\ \ZZ^{\mu_0+\mu_1-a} &  & 0 \end{array}  } \ar@{>>}[r] & {\begin{array}{c}\\ \\ H_{n-4}(\calM;\ZZ)\\ \rotatebox{90}{$\cong$} \\ 0 \end{array}}}
\]
where $T$ is a torsion group without $2$-torsion. We prove now that $T=0$. We have to deal separatedly with the cases
$corank(H(f))\geq 3$ and $corank(H(f))=2$.

\begin{lema}\label{t0e0}
If $corank(H(f)(O))\geq 3$, then $T=0$.
\end{lema}
\begin{proof}
Let $F:\CC^n\times S\to\CC$ be the unfolding associated with $f$ in Section~\ref{sectionunfoldings}. By 
Theorem~\ref{teounfolding1} there is a monodromy representation
\[\rho:\pi_1(S\setminus\Delta)\to Aut(H_{n-2}(F_s^{-1}(\delta)\cap B_\epsilon;\ZZ)).\]
By Lemma~\ref{connectedcover}, if one of the generators of the form $S_i$ of $H_{n-2}(\calX;\ZZ)$ maps to zero in $H_{n-2}(F_s^{-1}(\delta)\cap B_\epsilon;\ZZ)$,
then every other generator of the form $S_j$ maps to zero too. In the proof Lemma~\ref{cork3e0} we have seen that this is the case. By Homology Splitting we conclude that any $S_i$ is zero 
in $H_{n-2}(\calM;\ZZ)$.

Now let $z\in H_{n-2}(\calM;\ZZ)$ be a $p$-torsion element with $p\neq 2$. Then $pz=0$, which means that, considered with coefficients in $\ZZ_2$ its class $[z]\in H_{n-2}(\calM;\ZZ_2)$
must be also zero. This means that $z$ is homologous to $z'=\sum_i2a_iZ_i$. But from the exactness of the sequence
$$
0\to \oplus_iH_{n-2}(\calA_i;\ZZ)\to H_{n-2}(\calX;\ZZ)\to \oplus_iH_{n-3}(\partial\calA_i;\ZZ)\to H_{n-3}(\calB;\ZZ)\to 0
$$ we get that $2Z_i$ can be expressed as a linear combination of the $S_i$'s (recall that the $S_i$ are the images of the generators of $H_{n-2}(\calA_i;\ZZ)$,
and that the $Z_i$ corresponds to the generators of the kernel of $\oplus_iH_{n-3}(\calA_i\cap B;\ZZ)\to H_{n-3}(\calB;\ZZ)$, which is isomorphic to $\ZZ_2^{a-1}$). We can finally conclude that $z'$ can be expressed as a sum of some $S_i$'s, but as we have seen before, all of them are zero in $H_{n-2}(\calM;\ZZ)$.
\end{proof}

\begin{lema}\label{t0e1}
If $corank(H(f)(O))=2$, then $T=0$.
\end{lema}
\begin{proof}
Let $z\in H_{n-2}(\calM;\ZZ)$ be a $p$-torsion element with $p\neq 2$. Then we have the following equality with coefficients in $\ZZ_2$:
\[0=[z]\in H_{n-2}(\calM;\ZZ_2).\] 
As before, this means that homologically, $z$ can be expressed as $z=\sum_ia_iS_i$. Assume that all $S_i$ are equal in
$H_{n-2}(\calM;\ZZ)$. We would have that, integrating against the form $\Omega$ of Lemma~\ref{casosimplee1} and
Lemma~\ref{cork2e1} (normalizing it if necessary) we get
$$
\int_z \Omega=\sum_ia_i\int_{S_i} \Omega =\sum_ia_i
$$
which, by the hypothesis of $z$ being of $p$-torsion, means that $\sum_i a_i=0$, and, hence, that $[z]=0$.
 
We only need to prove that $S_i$ and $S_j$ represent the same class in $H_{n-2}(\calM;\ZZ)$ for any $i,j$.

If the functions 
\begin{equation}
\label{enefunciones}
\{h_{1,1},h_{1,2},h_{2,2},g_1,...,g_{n-3}\}
\end{equation}
form an i.c.i.s at the origin of Milnor number $0$ (that
is they are transverse) then there is only one sphere $S_1$ and the result is proved. Let us assume that they form an
i.c.i.s at the origin of Milnor number at least $1$.

Given a point $s_0\in S\setminus\Lambda$ there is a $1-1$ correspondence between points $p_i$ of $\Sigma[2]_{s_0}$ and 
spheres $S_i$ as above. To a vanishing cycle $\{p_i,p_j\}$ (recall Definition~\ref{vanishcycle}) corresponds a pair of
spheres $\{S_i,S_j\}$. By Lemmas~\ref{fullequivalence}~and~\ref{monodromiatransitiva} in order to prove that 
$S_i$ and $S_j$ represent the same class in $H_{n-2}(\calM;\ZZ)$ for any $i,j$ it is enough to show that there exists 
a vanishing cycle $\{p_i,p_j\}$ such that $S_i$ and $S_j$ represent the same class in $H_{n-2}(\calM;\ZZ)$. This reduces
the proof to the case in which the Milnor number of the i.c.i.s. defined by~(\ref{enefunciones}) at the origin is $1$.

The fact that the functions~(\ref{enefunciones}) have Milnor number $1$ at the origin implies that at least $n-1$ of
them must be linearly independent variables (after a suitable change of coordinates). After this it is easy to see that 
we can restrict ourselves to one of the following cases that we will list and analyse below. In this analysis we will 
use repeatedly the following fundamental fact, which is clear from Homology Splitting and from Section~\ref{sectiondecomp}: 

\textbf{Fact 1}. The homology of the Milnor Fibre of a germ $f$ only depends on the number of Morse points appearing in a 
generic value of $s$ of the base space of the versal deformation $S$ and on the topology of the triple 
$(\Sigma_s,\Sigma[1]_s,\Sigma[2]_s)$. The homology of $\calM$ only depends on the topology of the triple. The homology
of the Milnor Fibre has torsion if and only if the homology of $\calM$ has torsion.

The list of cases is the following:

\textsc{Case 1}. Suppose $f=(g_1,g_2)\cdot\left(\begin{array}{cc} g_3 & g_4 \\ g_4 & g_5\end{array}\right) \cdot \left(\begin{array}{c} g_1 \\ g_2\end{array}\right)$ with $g_1,g_2,g_3,g_4$ independent variables.
In this case, we can take coordinates such that $g_i=x_i$ for $i=1,\ldots,4$, and $g_5=ax_3+bx_5^2+\phi$, being $\phi$ a sum of higher order terms.

Consider the following family of functions: 
\[f_t=(x_1,x_2)\cdot\left(\begin{array}{cc} x_3 & x_4 \\ x_4 & ax_3+bx_5^2+t\phi\end{array}\right) \cdot \left(\begin{array}{c} x_1 \\ x_2\end{array}\right).\] 
It is clear that $f_1=f$.
For any $t$ the singular set $\Sigma$ is smooth, the set $\Sigma[1]$ is the surface given by the suspension of two smooth branches 
with intersection multiplicity equal to $2$, and the set $\Sigma[2]$ is just the origin.
After a perturbation the triple $(\Sigma,\Sigma[1],\Sigma[2])$ becomes a triple which has the topology of
\[(\CC^3,V(z_1(z_1+z_2^2-1)+z_3^2,V(z_1,z_2^2-1,z_3))\]
independently of $t$. Moreover in the generic perturbation
there are no $A_1$ points appearin outside $\Sigma$ for any $t$. Therefore, by Fact 1 in order to compute the homology
 of the Milnor fibre we may assume $t=0$.

Write $f_0=x_1^2x_3+2x_1x_2x_4+ax_3x_2^2+bx_2^2x_5^2=(x_1^2+ax_2^2)x_3+(2x_1x_2)x_4+bx_2^2x_5^2$. Since it is quasi-homogenous, we can take infinite Milnor radius and we are reduced to compute the homology of:
$$
(x_1^2+ax_2^2)x_3+(2x_1x_2)x_4+bx_2^2x_5^2=1.
$$
Projecting to $(x_1,x_2)$, we see that there exists a preimage if and only if $(x_1^2+ax_2^2,x_1x_2,bx_2^2)\neq(0,0,0)$, that is, everywehere except in the point $(x_1,x_2)=(0,0)$. It can be easily checked that the fibre over each point is contractible, and hence the Milnor fibre $F_{f_0}$ has the homotopy type of $\CC^2\setminus\{0\}\approx \SSS^3$.

\textsc{Case 2}. Suppose $f=(g_1,g_2)\cdot\left(\begin{array}{cc} g_3 & g_4 \\ g_4 & g_5\end{array}\right) \cdot \left(\begin{array}{c} g_1 \\ g_2\end{array}\right)$ with $g_1,g_2,g_3,g_5$ independent variables. We can write $f=(x_1,x_2)\cdot\left(\begin{array}{cc} x_3 & g_4 \\ g_4 & x_5\end{array}\right) \cdot \left(\begin{array}{c} x_1 \\ x_2\end{array}\right)$, where $g_4=ax_1+bx_2+x_4^2+\phi$, being $\phi$ again a sum of higher order terms. After an appropriate change of basis in $x_1$ and $x_2$ we get 
$$
f=(x_1,x_2)\cdot\left(\begin{array}{cc} x_3 & ax_1+bx_2+x_4^2+\phi \\ ax_1+bx_2+x_4^2+\phi & x_5\end{array}\right) \cdot \left(\begin{array}{c} x_1 \\ x_2\end{array}\right)=
$$
$$
=(x_1-x_2,x_2)\cdot\left(\begin{array}{cc} x_3 & ax_1+bx_2+x_3+x_4^2+\phi \\ ax_1+bx_2+x_3+x_4^2+\phi & 2ax_1+2bx_2+x_3+x_5\end{array}\right) \cdot \left(\begin{array}{c} x_1 -x_2\\ x_2\end{array}\right)=
$$
$$
=(x_2,x_1-x_2)\cdot\left(\begin{array}{cc} 2ax_1+2bx_2+x_3+x_5 & ax_1+bx_2+x_3+x_4^2+\phi \\ ax_1+bx_2+x_3+x_4^2+\phi &   x_3\end{array}\right) \cdot \left(\begin{array}{c} x_2 \\x_1 -x_2\end{array}\right)
$$
which falls into the previous case.

\textsc{Case 3}. Suppose $f=(g_1,g_2)\cdot\left(\begin{array}{cc} g_3 & g_4 \\ g_4 & g_5\end{array}\right) \cdot \left(\begin{array}{c} g_1 \\ g_2\end{array}\right)$ with
$g_1$ and $g_2$ are not linearly independent variables. After a change of base, we may assume that $f$ is of the form
$$
f=(x_1,q)\cdot\left(\begin{array}{cc} x_3 & x_4 \\ x_4 &   x_5\end{array}\right) \cdot \left(\begin{array}{c} x_1 \\q\end{array}\right)
$$
where $q$ has a Taylor development starting by a generic cuadric. Like in Case 1, using Fact 1 and an apropiate family
$f_t$, we may reduce the to the case in which $q=x_1^2+x_2^2+x_3^2+x_4^2+x_5^2$.

The triple $(\Sigma,\Sigma[1],\Sigma[2])$ and its deformations $(\Sigma_s,\Sigma[1]_s,\Sigma[2]_s)$ when we move $s$ in the base $S$ of the unfolding are always contained in the hyperplane $x_1=0$. We restrict to this hyperplane and forget 
the variable $x_1$ for the rest of the analysis of this case.

In this hyperplane, the i.c.i.s. $\Sigma$
is given the hypersurface $q=0$, and the singular locus of $det=x_3\cdot x_5-x_4^2$ is the $x_2$-axis. 
When we consider the Milnor Fibre $q^{-1}(s)$, it intersects the $x_2$ axis in two points. This two points correspond to two vanishing cyles $S_1,S_2$ in the Milnor Fibre $\calF$ 
of $\Sigma[1]=V(x_1,q,det)$. 
Each vanishing cycle $S_i$ corresponds to a point $p_i\in\Sigma[2]_s$, which gives a class $\calS_i$ in $H_3(\calM;\ZZ)$. Wee need to prove that these two classes are equal.
Running in this particular case the general considerations previously made
in order to compute the homology of $\calM$, we observe that if we find a vanishing cycle $S_3$ in $\calF$ 
meeting each $S_1$ and $S_2$ transversely at a point, we can use it and the fibrations above it, in order to express
the chain $S_1-S_2$ as a boundary.

The critical locus of the germ $(q,det):\CC^4\to\CC^2$ consists of four linear components, whose parametrizations are 
given by $(t,0,0,0)$ $(0,t,0,t)$, $(0,t,0,-t)$ and $(0,0,t,0)$ respectively. The corresponding components of the
discriminant are parametrized as follows: $(t,0)$, $(2t^2,t^2)$, $(2t^2,-t^2)$ and $(t^2,-t^2)$. Since we are working on the milnor fibre of $q$, we are looking at the preimage of the set $\{(x,y)\in \CC^2\mid x=1\}$. In that line, the point $(1,0)$ correspond to the values where we want to look for the vanishing cycle touching the two critical points, which are $(1,0,0,0)$ and $(-1,0,0,0)$. In order to track how this cycle vanishes, we will consider the interval $(1,\epsilon)$, where $\epsilon$ ranges from $0$ to $\frac{1}{2}$. We will consider the expansion of $q$ and $det$ based in the point $(0,\frac{1}{\sqrt{2}},0,\frac{1}{\sqrt{2}})$:
$$
q=x_2^2+\sqrt{2}x_3+x_3^2+x_4^2+\sqrt{2}x_5+x_5^2+1
$$
$$
det=\frac{1}{2}+\frac{1}{\sqrt{2}}(x_3+x_5)+x_3x_5-x_4^2.
$$
For a fixed $\epsilon\in[0,\frac{1}{2}]$, the fibre over the point $(1,\epsilon)$ is given by
$$
\frac{1}{2}w^2+\sqrt{2}w+x_2^2+\frac{1}{2}z^2+x_4^2=0
$$
$$
x_2^2+3x_4^2+z^2=1-2\epsilon
$$
where $w=(x_3+x_5)$, $z=x_3-x_5$.

The real solutions of $x_2^2+3x_4^2+z^2=1-2\epsilon$ are a single point if $\epsilon=\frac{1}{2}$ and a 2-sphere if $\epsilon\in[0,\frac{1}{2})$. 
Fixed $x_2$, $x_4$ and $z$, there are two posible choices for $w$, except when the discriminant of $\frac{1}{2}w^2+\sqrt{2}w+x_2^2+\frac{1}{2}z^2+x_4^2$ vanishes,
that is, when $x_2^2+\frac{1}{2}z^2+x_4^2=1$. But this condition, togeteher with $x_2^2+3x_4^2+z^2=1-2\epsilon$ implies $4x_4^2+z^2=-4\epsilon$, 
which does not have real solutions if $\epsilon>0$. Since $\SSS^2$ is simply connected, the only possible double cover over it is two copies of $\SSS^2$.
That is, we have two copies of $\SSS^2$ over each point between $(1,0)$ and $(1,\frac{1}{2})$; this two spheres collapse when we go to $(1,\frac{1}{2})$,
and they intersect in two different points at $(1,0)$. This two points of intersection are preciselly $(1,0,0,0)$ and $(-1,0,0,0)$, which are the singular points of $det$ at $q=1$.
Any of this two spheres is a vanishing cycle as we are looking for.

\textsc{Case 4}. If $f$ is of the form
$$
f=(x_1,x_2,g_6)\cdot\left(\begin{array}{ccc} x_3 & x_4 & 0 \\ x_4 & x_5 & 0 \\ 0& 0 & 1\end{array}\right) \cdot \left(\begin{array}{c} x_1 \\ x_2 \\ g_6\end{array}\right)
$$
with the linear part of $g_6$ lineally dependent with $x_1,x_2,x_3,x_4,x_5$, the configurations
$(\Sigma,\Sigma[1],\Sigma[2])$ and its deformations $(\Sigma_s,\Sigma[1]_s,\Sigma[2]_s)$ 
are easily checked to be suspensions of those in the previous case. Since all the method depends on this configuration,
this case can be treated in the same way as the previous one.
\end{proof}

\subsection{The case of $corank(H(f)(O))=1$}

In this case $\calX=\calB$ fibres over $h_t^{-1}(0)\approx \vee_{\mu_1}\SSS^2$ with fibre $\SSS^{n-3}$, and $\calB_u$ fibres over $h_t^{-1}(0)$ with fibre $\SSS^{n-4}$.
Since $h_t^{-1}(0)$ is simply connected, both fibrations are orientable. Using the Gysin sequence of these fibrations we get that $H_k(\calB;\ZZ)\cong\ZZ$ for $k=n-3,0$,
$H_k(\calB;\ZZ)\cong\ZZ^{\mu_1}$ for $k={n-1},2$, and $0$ otherwise. Adding the Lefschetz thimbles as before, we obtain that 
\[H_{n-1}(\calM;\ZZ)\cong\ZZ^{2\mu_1+\mu_0},\]
\[H_{n-3}(\calM;\ZZ)\cong\ZZ,\]
\[H_2(\calM;\ZZ)\cong \ZZ^{\mu_0}\]
\[H_0(\calM;\ZZ)\cong\ZZ,\]
and the rest of the homology groups are trivial.

\section{The homology of the Milnor fibre}
\label{sechomolfibramilnor}

After we have computed the homology of $\calM$ we can use Proposition~\ref{homsplit} in order to compute the homology of the Milnor fibre of $f$.

Since the tubular neighbourhood $T$ is homotopy equivalent to the Milnor fibre of $\Sigma_s$ of the $3$-dimensional i.c.i.s. $\Sigma_0$ we have 
\[H_0(T;\ZZ)\cong\ZZ\]
\[H_3(T;\ZZ)\cong\ZZ^{\mu_0}\]
\[H_i(T;\ZZ)= 0\] for any other $i$.

The inclusion of $\calM$ in $T$ gives clearly an isomorphism in the $H_3$ when $n\geq 7$,
and hence $H_i(T,\calM;\ZZ)=0$ for $1\leq i\leq 3$, and $H_{i+1}(T,\calM;\ZZ)\cong H_i(\calM;\ZZ)$ for $i\geq 4$. 

We have obtained:
\begin{theo}
\label{homologia}
Let $\mu_0$ and $\mu_1$ be the Milnor numbers of the i.c.i.s. $(g_1,...,g_{n-3})$ and $(det(H(f)),g_1,...,g_{n-3})$.
The homology of the Milnor fibre is the following:
\begin{itemize}
\item If $corank(H(f)(0)\geq 3$:
\[H_{n-1}(\mathbf{F}_f;\ZZ)\cong \ZZ^{\mu_0+2\mu_1-4a+1+\#A_1},\]
\[H_{k}(\mathbf{F}_f;\ZZ)=0\]
if $1\leq k\leq n-2$,
\[H_0(\mathbf{F}_f;\ZZ)\cong \ZZ.\]
\item If $corank(H(f)(0)=2$:
\[H_{n-1}(\mathbf{F}_f;\ZZ)\cong \ZZ^{\mu_0+2\mu_1-4a+2+\#A_1},\]
\[H_{n-2}(\mathbf{F}_f;\ZZ)\cong \ZZ,\]
\[H_{k}(\mathbf{F}_f;\ZZ)=0\]
if $1\leq k\leq n-3$,
\[H_0(\mathbf{F}_f;\ZZ)\cong \ZZ.\]
\item If $corank(H(f)(0)=1$:
\[H_{n-1}(\mathbf{F}_f;\ZZ)\cong\ZZ^{\mu_0+2\mu_1},\]
\[H_{n-3}(\mathbf{F}_f;\ZZ)\cong\ZZ,\]
\[H_{k}(\mathbf{F}_f;\ZZ)=0\]
if $k=n-2$ and if $1\leq k\leq n-4$,
\[H_0(\mathbf{F}_f;\ZZ)\cong \ZZ.\]
\item If $corank(H(f)(0)=0$:
\[H_{n-1}(\mathbf{F}_f;\ZZ)\cong\ZZ^{\mu_0},\]
\[H_{n-4}(\mathbf{F}_f;\ZZ)\cong\ZZ,\]
\[H_{k}(\mathbf{F}_f;\ZZ)=0\]
if $k=n-2, n-3$ and $1\leq k\leq n-4$,
\[H_0(\mathbf{F}_f;\ZZ)\cong \ZZ.\]
\end{itemize}
\end{theo}
\begin{proof}
Our computations work for the case $corank(H(f)(0))>0$, if $n\geq 8$. In order to remove this restrictions we notice that by Thom-Sebastiani the Milnor fibre of $f+z^2$
with $z$ a new variable is the suspension of the original Milnor fibre, and that the case $corank(H(f)(0))=0$ was proved by Nemethi~in~\cite{Nm}.
\end{proof}

\section{The homotopy type of the Milnor fibre}
\label{sechomotfibramilnor}

\begin{prop}
\label{fungroup}
The Milnor fibre $\mathbf{F}_f$ is simply connected if $corank(h_{i,j}(0))\neq 0$.
\end{prop}
\begin{proof}
For $n\geq 6$, the Kato-Matsumoto bound~\cite{KM} tells us that $\mathbf{F}_f$ is simply connected. 
For the case where $n=5$, we will need the following reasoning. 

Let $\calZ_1,...,\calZ_{\#A_1}$ be representatives of the vanishing cycles of $\mathbf{F}_f$ corrsponding to the 
$A_1$ points that appear outside $\Sigma_s$ in a generic deformation. Let $C(\calZ_i)$ denote the cone over $\calZ_i$. 
Let $C(\pi)$ be the cylinder of the mapping 
\[\pi:\calM\to\Sigma_s\]. The space $C(\pi)$ is simply connected because it admits the simply connected space $\Sigma_s$ as a deformation retract.

By construction we have that 
\[\mathbf{F}_f\cup C(\pi)\cup_{\#A_1}C(\SSS^4)\]
is homotopy equivalent to the contractible space $X_s$ (see Section~\ref{sechomsplit}).       
Since each $\calZ_i$ is homeomorphic to $\SSS^4$, by Seifert-Van Kampen theorem, the gluing of the $C(\calZ_i)$ has no efect over the fundamental group, 
since both $\pi_1(C(\SSS^4))$ and $\pi_1(\SSS^4)$ are trivial.
The same reasoning tells us that, if $\pi_1(\calM)$ is trivial, so must be $\pi_1(\mathbf{F}_f)$.

The space $\calM$ is obtained from $\calX$ by gluing the preimage by $\pi$ of several Lefschetz thimbles. These pieces are topologically $D^3\times \SSS^1$ glued along 
$\SSS^2\times \SSS^1$. By Seifert-Van Kampen theorem, if $\pi_1(\calX)$ is trivial, the adition of these pieces does not change the fundamental group.
So, to prove that $\pi_1(\calM)=0$ it is enough to prove that  $\pi_1(\calX)=0$.

We may compute $\pi_1(\calX)$ using Seifert-Van Kampen with the decomposition 
\[\calX=\calB\cup \calA_1\cup\cdots,\cup\calA_a.\]
In Section~\ref{sectiondecomp} it is shown how the mapping $\pi$ allows to express each of the pieces of the
decomposition as fibrations with fibres homotopy spheres of dimension at leat $2$ over the corresponding piece of the decomposition 
\[\Sigma_s\cap det(H(F_s))^{-1}(0)=B_0\cup A_1(\zeta)\cup\cdots\cup A_a(\zeta).\] 
Using this it is easy to see that the computation of $\pi_1(\calX)$ by Seifert van Kampen mimics the computation of
$\pi_1(\Sigma_s\cap det(H(F_s))^{-1}(0))$, but this space is simply connected (in fact a bouquet of $2$-spheres). 
\end{proof}

We now have all the necessary ingredients to prove our Bouquet Theorem.

\begin{theo}
\label{homotopiabouquet}
The Milnor fibre of a singularity over a $3$-dimensional i.c.i.s. with finite extended codimension has the homotopy
type of a bouquet of spheres of different dimensions.
\end{theo}
\begin{proof}
For the previous Proposition, we know that the Milnor fibre is simply connected.

In the case where $corankH(f)(O)\geq 2$ (that is, $a\neq 0$) we have computed the integer homology, getting that $H_{n-1}(\mathbf{F}_f;\ZZ))$ and $H_{n-2}(\mathbf{F}_f;\ZZ)$ are free and finitely generated and 
$H_i(\mathbf{F}_f;\ZZ)\cong 0$ otherwise. In this situation, since the Milnor Fibre has the homotopy type of a $(n-1)$-complex, we can apply Criterion 2.2 and 
Remark 2.3 of~\cite{Nm} and we get the result.

If $corank(H(f)(O)=0$ the result is covered by Theorem~4.1~of~\cite{Nm}. 

We are left with the case in which $corank(H(f)(O)=1$.
By Criterion 2.2 in \cite{Nm}, we only need to represent each generator of the non-zero homology groups by a chain 
modelled in a sphere. When $corank(H(f)(O)=1$, in the decomposition of $\calM$ given in Section~\ref{sectiondecomp} we have that $\calB$
coincides with $\calX$, that $B_u$ is diffeomorphic to $B_0$, which are Milnor fibres of the $2$-dimensional
i.c.i.s. $\Sigma_0\cap V(det(H(f)))$ and that the fibration~(\ref{fibracsn-3}) becomes a homotopy $\SSS^{n-3}$-fibration
\begin{equation}
\label{fibsimple}
\varphi:\calX\to B_0\cong \calB_u.
\end{equation}

The generator of $H_{n-3}(\mathbf{F}_f;\ZZ)$ is the Gysin lift of the generator of $H_0(B;\ZZ)$, and hence it is represented by a sphere. By Homology Splitting, the generators of $H_{n-1}(\mathbf{F}_f;\ZZ)$ come from
two different places: the ones comming from the $A_1$-singularities of $f_s$ outside $\Sigma_s$ and those coming from 
$H_{n-1}(\calM;\ZZ)$. The first generators are clearly represented by spheres (the vanshing cycles of the 
$A_1$-singularities). The generators of $H_{n-1}(\calM;\ZZ)$ come in turn from two different places:
the ones comming from the image of $H_{n-1}(\calX;\ZZ)$ in $H_{n-1}(\calM;\ZZ)$, and 
those coming from the addition to $\calX$ of the spaces $\calC_i$ (see the decomposition formula~(\ref{descom})). Recall
that each $\calC_i$ is the product of a Lefschetz thimble associated to a vanishing cycle of $B_u=\{det(H(f_s)=u\}\cap\Sigma_s$ with the homotopy-sphere $\SSS^{n-4}$, which is the fibre of the
fibration~(\ref{fibracionsn-4}). The ones coming from $H_{n-1}(\calX;\ZZ)$ are Gysin-liftings over the vanishing cycles
of $B_u$ of the fibration~(\ref{fibsimple}).

We claim that the fibration of $(n-3)$-spheres over $B_0$ is trivial. Since $B_0$ is a bouquet of $2$-spheres given by
vanishing cycles it is enough to prove that the fibration, restricted to each of the vanishing cycles of $B_0$ 
is trivial. Choose a vanising cycle $C_i$. Move the parameter $s$ so that that $s$ is very close to a parameter
$s_0$ in which $\Sigma_s\cap V(det((f_s)))$ adquires an $A_1$ singularity to which the vanishing cycle $C_i$ collapses.
In this situation a local change of coordinates shows that to prove that the fibration is trivial over $C_i$ is equivalent to prove that
the fibration of $(n-3)$-spheres associated to the function
\[f=(x_1^2+x_2^2+x_3^2)x_4^2+\sum_{i=5}^nx_i^2\]
is trivial over the vanishing cycle of the restriction of $x_1^2+x_2^2+x_3^2$ to $V(x_4,...,x_n)$. Proving this is an 
easy local computation.

Now we represent each of the two kinds of generators of $H_{n-1}(\calM;\ZZ)$ by spheres.
Let us start by the first kind. By the claim the group $H_{n-1}(\calX;\ZZ)$ is generated by chains of the form 
\[\tau:\SSS^2\times \SSS^{n-3}\to\calX\subset\calM\subset\mathbf{F}_f,\]
where $\tau(\SSS^2\times \SSS^{n-3})$ is a Gysin lift of a canishing cycle $C_i$ of $B_0$ by the fibration~(\ref{fibsimple}).

Choose a section $s$ of this fibration such that $s(C_i)$ is inside $\tau(\SSS^2\times\SSS^{n-3})$. 
For $n=5$, the sphere $s(\SSS^2)$ is trivial in $H_2(\calX;\ZZ)$,
since this group is generated by the fibre. This implies that it is also zero in $H_2(\mathbf{F}_f;\ZZ)$, and, by Hurewitz's Theorem, it is also trivial in $\pi_2(\mathbf{F}_f)$. 
For $n>5$ the triviality of $s(\SSS^2)$ in $\pi_2(F_t)$ holds by
the connectivity of the Milnor fibre. This means that $s(\SSS^2\times\{point\})$ can be killed by a $3$-disc inside $\mathbf{F}_f$. By Lemma 4.5~in~\cite{Nm}, we have that the homology class $[\tau(\SSS^2\times\SSS^{n-3})]$ can be 
represented by a sphere.

We study now the homology classes in $H_{n-1}(F_1;\ZZ)$ coming from a the addition of an space $\calC_i$. 
The space $\calC_i$ is the product of a Lefschetz thimble $L_i$ associated to a vanishing cycle $C_i$ of $B_u$ with the
sphere $\SSS^{n-4}$, which is the homotopy-fibre of the fibration~(\ref{fibracionsn-4}). 
Recall that over $B_u$ we have in fact a fibration of pairs with fibre homotopic to $(\SSS^{n-3},\SSS^{n-4})$ being 
$\SSS^{n-4}$ embedded as the equator of $\SSS^{n-3}$. Consider a collar $K\cong\partial L_i\times [0,1]$ of 
$\partial L_i$ in the $3$-cell $L_i$. We deform continuously 
the chain given by the embedding of $L\times\SSS^{n-4}$ in $\calM$ so that fibrewise $\SSS^{n-4}$ is the equator of $\SSS^{n-3}$ over any point of the internal boundary of
the collar and so that $\SSS^{n-4}$
is collapsed to the north pole of $\SSS^{n-3}$ at the external boundary $\partial L_i$ of the collar.
The obtained chain is called
\[\varphi:L_i\times\SSS^{n-4}\to\mathbf{F}_f.\]
The mapping 
\[s:\partial L_i\to\calX\subset\mathbf{F}_f\]
which assigns to a point of $\partial L_i$ the north pole of the fibre $\SSS^{n-3}$ has been seen before to be a trivial element in
$\pi_2(\mathbf{F}_f)$. Therefore there exists a $3$-disk
$L'$ bounding $\partial L_i$ and an extension
\[s':L'\to F_t\]
of $s$.
The identification $L\cup_{\partial L_i}L'$ along their common boundary is a $3$-sphere. A representative of our homology class is given by the chain
\[\psi:(L_i\cup_{\partial L_i}L')\times\SSS^{n-4}\to F_t\]
defined by $\psi|_{L\times\SSS^{n-4}}:=\varphi$ and $\psi|_{L'\times\SSS^{n-4}}:=s'\comp pr_1$, where $pr_1$ is the projection of $L'\times\SSS^{n-4}$ to the first factor.
Notice that the source of $\psi$ is a product of spheres, which we view as a trivial fibration of $\SSS^{n-4}$ over $L \cup_{\partial L_i}L'\cong\SSS^3$,
and that $\psi$ fatorises through the result of collapsing to a point the fibre over any point of $L'$. Again Lemma~4.5~in~\cite{Nm} represents the homology class by a sphere.
\end{proof}

\section{Examples}
\label{sectionexamples}
Despite the apparent simplicity of the homotopy type of the Milnor fibre of the class singularities considered in 
this paper, it is possible to find among them unexpected topological behaviours which at the moment have not been 
observed in singularities with smaller critical set. As an illustration of this we summarise here the properties of 
a family of examples, which fall in the general class studied in this paper, and which was used in~\cite{FdB3} to produce
counterexamples to several old equisingularity questions.

\begin{example}
\label{contraejemplos}
Let $\varphi$ a possibly identical to $0$ convergent power series in a variable $x_1$. Define
\[f_\varphi:(\CC^5,O)\to\CC\]
by
\[f_\varphi(x_1,x_2,x_3,y_1,y_2):=f=(y_1,y_2)\cdot\left(\begin{array}{cc} x_3 & x_2 \\ x_2 & \varphi(x_1)-x_3\end{array}\right) \cdot \left(\begin{array}{c} y_1 \\ y_2\end{array}\right).\]
\end{example}

If $\varphi$ is not identical to $0$ the function $f_\varphi$ is of finite codimension with respect to the ideal 
$I=(y_1,y_2)$. The critical set $\Sigma=V(y_1,y_2)$ is $3$-dimensional and smooth. It is easily checked that the 
$I$-unfolding
\begin{equation}
\label{exampleunfolded}
F_\varphi:=f_\varphi+\sum_{i=0}^{ord(\varphi)-2}t_ix_1^iy_2^2,
\end{equation}
where $ord(\varphi)$ denotes the order of the series $\varphi$ in $x_1$,
is the versal $I$-unfolding of $f_\varphi$ in the sense of~\cite{Pe}~and~\cite{FdB2}. Hence we can obtain all
$I$-unfoldings of $f_\varphi$ by considering deformations of the form
\[\varphi+\sum_{i=0}^{ord(\varphi)-2}t_ix_1^i.\]

Notice that the determinant
\[detH(f_\varphi)=x_3(\varphi(x_1)-x_3)-x_2^2:(\Sigma,O)\to\CC\] 
has a singularity at the origin of type $A_{2ord(\varphi)-1}$. An easy computation shows that if $(f_\varphi)_s$ is a 
generic deformation of $f_\varphi$ in its versal $I$-unfolding, the cardinality of the set $\Sigma[2]_s$ of points where
$H(f_2)$ has corank precisely $2$ is equal to $ord(\varphi)$. 

It is also easy to check that for any $s$ in the base of the versal $I$ unfolding the critical set of $f_s$ is 
equal to $\Sigma=V(y_1,y_2)$. Hence there are no $A_1$ points popping out of $\Sigma$ in a generic $I$-unfolding of
$f_\varphi$. 

Noticing that $corank(H(f_\varphi))(O)=2$ we may apply Theorem~\ref{homologia} to show that the Milnor fibre is $2$-connected,
with third Betti number equal to $1$ and fourth Betti number equal to:
\[b_4=\mu_0+2\mu_1-4a+2+\#A_1=0+2(2ord(\varphi)-1)-4ord(\varphi)+2+0=0,\] 
which, surprisingly, is independent of $\varphi$. By Theorem~\ref{homotopiabouquet} we conclude that the Milnor
fibre of $f_\varphi$ is homotopy equivalent to a $3$-sphere. The remarkable fact is that 
the homotopy type of the Milnor fibre is independent on $\varphi$ and at the same time the topology of the pair of germs
\begin{equation}
\label{pardegermenes}
((\Sigma,O),(\Sigma[1]_s,O))
\end{equation}
depends heavily on the value $s$ in the base of the versal unfolding.

In~\cite{FdB3} it is shown that in fact the diffeomerphism type of the Milnor fibration of the germ $f_{\varphi}$ and the generic 
L\^e-numbers are independent of $\varphi$. Using that the topology of the pair~(\ref{pardegermenes}) depends on $s$ it is also proven that
the topology of the abstract link
of $f_\varphi$ does depend on $\varphi$. This kind of examples and their stabilisations are at the moment the only known
families of examples with constant L\^e numbers and constant Milnor fibration and changing topological type. They 
answer negatively a question 
of D. Massey in~\cite{Ma}. In~\cite{FdB3} modifications of these examples are also used to give the first known 
counterexample of Zariski's Question~B~of~\cite{Zar}. Also in~\cite{FdB3} these examples were used
to construct a family of reduced projective hypersurfaces with
constant homotopy type and changing topological type (therefore most classical algebro-topological invariants can not detect
the change in topology).


\begin{thebibliography}{99}
\bibitem{AGV} V.I. Arnolʹd, S.M. Guseĭn-Zade, A.N. Varchenko. {\em Singularities of differentiable maps. Vol. II. Monodromy and asymptotics of integrals.} 
Monographs in Mathematics, \textbf{83}. Birkhäuser, (1988).
\bibitem{FdB1} J. Fern\'andez de Bobadilla. {\em Approximations of non-isolated singularities of finite codimension with respect to an isolated complete intersection singularity}.
Bull. London Math. Soc. \textbf{35}, (2003), 812-816.
\bibitem{FdB2} J. Fern\'andez de Bobadilla. {\em Relative Morsification Theory}. Topology \textbf{43}, (2004), 925-982.
\bibitem{FdB3} J. Fern\'andez de Bobadilla. {\em Answers to some equisingularity questions}. Inventiones Math. \textbf{161}, (2005), 657-675.
\bibitem{FdB4} J. Fern\'andez de Bobadilla. {\em On homotopy types of complements of analytic sets and Milnor fibres}. ArXiv:0907.2176. 
\bibitem{Ga} T. Gaffney. {\em Invariants of $D(q,p)$ singularities}.  Real and complex singularities, Contemp. Math. \textbf{459}, Amer. Math. Soc., Providence, RI, (2008), 13--22
\bibitem{HL} H. Hamm, L\^e Dung Trang. {\em Local generalisations of Lefschetz-Zariski theorems}. J. Reine Angew. Math. \textbf{389}, (1988), 157-189.
\bibitem{dJ} T. de Jong. {\em Some classes of line singularities}. Math. Z. \textbf{198}, (1998), 493-517.
\bibitem{KM} M. Kato, Y. Matsumoto. {\em On the connectivity of the Milnor fiber of a holomorphic function at a critical point}.  Manifolds-Tokyo 1973 (Proc. Internat. Conf., Tokyo, 1973),
131-136. Univ. Tokyo Press, Tokyo, 1975.
\bibitem{Lo} E. Looijenga. {\em Isolated singular points on complete intersections}. London Mathematical Society Lecture Note Series, \textbf{77}. Cambridge University Press, Cambridge, (1984).
\bibitem{Ma} D. Massey. {\em The L\^e varieties II}. Invent. Math. \textbf{104} (1991), 113-148.
\bibitem{Mi} J. Milnor. {\em Singular points of complex hypersurfaces}. Annals of Math. Studies \textbf{61}. Princeton Univ. Press, (1968).
\bibitem{Ne} A. Nemethi. {\em The Milnor fiber and the zeta function of the singularities of type $f=P(h,g)$}.
Compositio Math. \textbf{79}, (1991), 63-97.
\bibitem{Nm} A. Nemethi. {\em Hypersurface singularities with 2-dimensional critical locus}. J. London Math. Soc., \textbf{59}, (1999), 922-938.
\bibitem{Pe} R. Pellikaan. {\em Finite determinacy of functions with non-isolated singularities}. Proc. London Math. Soc. (3), \textbf{57}, (1988) no. 2, 357-382.
\bibitem{Sh} M. Shubladze. {\em Isolated hypersurface singularities of the transversal type $A_1$}. 
Bull. Georgian Acad. Sci. \textbf{153}, (1996), no. 1, 7-10.
\bibitem{Si1} D. Siersma. {\em Isolated line singularities}.  Singularities, Part 2 (Arcata, Calif., 1981),  485-496, Proc. Sympos. Pure Math., \textbf{40},
Amer. Math. Soc., Providence, RI, 1983.
\bibitem{Si2} D. Siersma. {\em Hypersurfaces with singular locus a plane curve and transversal type $A_1$}.  Singularities (Warsaw, 1985),  397-410, Banach Center Publ., \textbf{20}, 
PWN, Warsaw, 1988. 
\bibitem{Si3} D. Siersma. {\em The vanishing topology of non isolated singularities}.  New developments in singularity theory (Cambridge, 2000),  
447-472, NATO Sci. Ser. II Math. Phys. Chem., \textbf{21}, Kluwer Acad. Publ., Dordrecht, 2001.
\bibitem{Si4} D. Siersma. {\em A bouquet theorem for the Milnor fibre}. J. Algebraic Geom. \textbf{4}, (1995), 51-66.
\bibitem{Ti} M. Tibar. {\em Bouquet decomposition of the Milnor fibre}. Topology \textbf{35} (1), (1996), 227-241.
\bibitem{Za} A. Zaharia. {\em Topological properties of certain singularities with critical locus a 2-dimensional complete intersection}. Topology Appl. \textbf{60} (1994), 153-171.
\bibitem{Zar} O. Zariski. {\em Open questions in the theory of singularities}. Bull. Amer. Math. Soc. \textbf{77} (1971), 481-489.
\end{thebibliography}
\end{document}